\documentclass{amsart}

\usepackage[utf8]{inputenc}
\usepackage{amsfonts}
\usepackage{comment}
\usepackage{xcolor}

\usepackage{amsthm,amsmath,amssymb}
\usepackage{mathrsfs}
\usepackage[numbers]{natbib}  
\usepackage[colorlinks,citecolor=blue,urlcolor=blue]{hyperref}  

\usepackage{enumerate, enumitem}

\usepackage{tgtermes}
\usepackage{mathptmx} 
\usepackage[scaled=.92]{helvet}

\date{\today}

\numberwithin{equation}{section}

\newtheorem{theorem}{Theorem}[section]
\newtheorem{lemma}[theorem]{Lemma}
\newtheorem{proposition}[theorem]{Proposition}
\newtheorem{corollary}[theorem]{Corollary}
\newtheorem{definition}[theorem]{Definition}
\newtheorem{example}[theorem]{Example}

\newtheorem{remark}[theorem]{Remark}

\newcommand{\N}{\mathbb{N}}

\newcommand{\Q}{\mathbb{Q}}
\newcommand{\R}{\mathbb{R}}

\newcommand{\hR}{\,\!^\ast\R}
\newcommand{\sh}[1]{\,\!^\circ#1} 

\newcommand{\civita}{\mathcal{R}}

\newcommand{\an}{C^{\omega}}

\newcommand{\ext}[2]{\overline{#2}_{#1}}

\newcommand{\rP}{\mathbb{P}}
\renewcommand{\P}{\mathcal{P}}

\newcommand{\norm}[1]{\left\Vert#1\right\Vert}

\newcommand{\p}[2]{p_{#1}^{#2}}

\DeclareMathOperator*{\wlim}{w-lim}
\DeclareMathOperator*{\slim}{s-lim}

\DeclareMathOperator*{\supp}{supp}
\newcommand{\m}{m}
\newcommand{\meas}{\mathcal{M}}

\newcommand{\leb}{\m_L}

\renewcommand{\L}{\mathcal{L}}
\newcommand{\sL}{\L_s}
\newcommand{\wL}{\L_w}

\newcommand{\s}{\overline{s}}

\begin{document}
	\title{Spaces of measurable functions on the the Levi-Civita field}
	\author{Emanuele Bottazzi}
	\address{Department of Civil Engineering and Architecture, University of Pavia,Via Adolfo Ferrata 3, 27100 Pavia, Italy}
	\email{emanuele.bottazzi@unipv.it, emanuele.bottazzi.phd@gmail.com}
	\maketitle
		
	
\begin{abstract}
	We introduce the $\L^p$ spaces of measurable functions whose $p$-th power is summable with respect to the uniform measure over the Levi-Civita field. These spaces are the counterparts of the real $L^p$ spaces based upon the Lebesgue measure.
	Nevertheless, they lack some properties of the $L^p$ spaces: for instance, the $\L^p$ spaces are not sequentially complete with respect to the $p$-norm.
	This motivates the study of the completions of the $\L^p$ spaces with respect to strong convergence, denoted by $\sL^p$.
	It turns out that the $\sL^p$ spaces are Banach spaces and that it is possible to define an inner product over $\sL^2$, thus making it a Hilbert space.
	Despite these positive results, these spaces are still not rich enough to represent every real continuous function.
	For this reason, we settle upon the representation of real measurable functions as sequences of measurable functions in $\civita$ that weakly converge in measure.
	We also define a duality between measurable functions and representatives of continuous functions. This duality enables the study of some measurable functions that represent real distributions. We focus our discussion on the representatives of the Dirac distribution and on the well-known problem of the product between the Dirac and the Heaviside distribution, and we show that the solution obtained with measurable functions over $\civita$ is consistent with the result obtained with other nonlinear generalized functions.
\end{abstract}

\tableofcontents

\section{Introduction}

The Levi-Civita field $\civita$, introduced by Levi-Civita in \cite{civita1, civita2} and subsequently rediscovered by many authors in the '900, is the smallest non-Archimedean ordered field extension of the field $\R$ of real numbers that is both real closed and sequentially complete \cite{analysislcf}.
Since the early '90s, Berz and Shamseddine have begun a fruitful development of analysis on the Levi-Civita field (see for instance \cite{berz,analysislcf,reexp,odes,Shamseddinephd} and references therein).
The study of the differential structure of the Levi-Civita field has also led to the implementation of effective algorithms for the computation of the derivatives of real functions as differential quotients \cite{computational,inf calc}.
The Levi-Civita field has also been used as the field of scalar for the implementation of algorithms of dynamic geometry based upon results obtained with techniques of nonstandard analysis \cite{strobel}.

Recently, a non-Archimedean measure on $\civita$ inspired by the Lebesgue measure has been developed by Shamseddine and Berz in dimension $1$ \cite{shamseddine2012,berz+shamseddine2003}, and by Shamseddine and Flynn in dimension $2$ and $3$ \cite{shamflin1}.
This measure shares some properties with the real Lebesgue measure: for instance, the length of an interval is equal to the difference of its endpoints, and a set is measurable if and only if it can be approximated by a sequence of intervals with an arbitrary precision.
However,
many properties of the real Lebesgue measure fail for its non-Archimedean counterpart in $\civita$.
A remarkable difference is that the measure on the Levi-Civita field is not defined on an algebra of sets, since there are some null sets whose complement is not measurable, as it does not contain any intervals of a finite measure.
Nevertheless, the non-Archimedean measure allows for the introduction of an integration theory over $\civita$.
Shamseddine and Berz showed that the set of all integrable functions over a measurable set is a $\civita$-vector space, and that some classical results of integration, such as the fundamental theorem of calculus, can be extended to the Levi-Civita field.

The purpose of our paper is to further develop the measure and integration theory on the Levi-Civita field and to introduce the spaces $\L^p$ of measurable functions whose $p$-th power is summable with respect to the uniform measure over the Levi-Civita field, in analogy with the real $L^p$ spaces based upon the Lebesgue measure.

In other non-Archimedean extensions of the real numbers, such as fields of hyperreal numbers of Robinson's framework of analysis with infinitesimals \cite{robinson} or the Robinson valuation fields $^\rho\R$, called also the fields of Robinson’s asymptotic numbers \cite{robi naa}, this goal can be achieved by extending the real $L^p$ spaces. For $\hR$ this extension is obtained by the star map, while for $^\rho\R$ it has been obtained by Pestov \cite{pestov} by exploiting the fact that $^\rho\R$ is a quotient of $\hR$ and that it is spherically complete \cite{lux} (for the notion of spherical completeness of a valued field, see \cite{nfa}).
In both cases, the resulting extension of the real $L^p$ spaces is a Banach space.
Since $\civita$ is an ordered and valuation subfield of $^\rho\R$, one might attempt at following a similar route for the introduction of the $\L^p$ spaces over the Levi-Civita field.
However, the results of Pestov cannot be adapted to $\civita$, since this field is only sequentially complete and it lacks the stronger topological property of spherical completeness.

Nevertheless, the $\L^p$ spaces can be introduced explicitly with a definition analogous to that of the real $L^p$ spaces based upon the Lebesgue measure.
While the $\L^p$ spaces share several properties with the real $L^p$ spaces, they are not sequentially complete with respect to the $p$-norm.
This limitation will motivate us to study the completions of the $\L^p$ spaces with respect to strong convergence, denoted by $\sL^p$.
These spaces are Banach spaces, moreover, it is possible to define an inner product over $\sL^2$, thus making it a Hilbert space.

As an application of the non-Archimedean measure, Flynn and Shamseddine proved measurable functions on the Levi-Civita field are general enough to represent some distributions \cite{shamflin2}.
For instance, it is possible to define some measurable functions that represent the Dirac distribution, much in the spirit of the representation of distributions with functions of nonstandard analysis \cite{benci_schwartz,grid functions,todorov}.
Notice however that Flynn and Shamseddine only established a duality between the so-called delta-like measurable functions and analytic functions over $\civita$.
Nevertheless, the authors were able to use these representatives of the Dirac distribution for the study of some simple differential equations in $\civita$ and in $\civita^3$. They have also suggested that measurable representatives of the Dirac distribution could be used to describe more complex phenomena.

These results remind of analogous approaches to the representation of distributions as measurable functions over non-Archimedean fields, usually obtained with techniques of Robinson's analysis with infinitesimals.
Some embeddings of distributions in algebras of functions of nonstandard analysis have been obtained for instance by Kinoshita \cite{kinoshita}, Hoskins and Pinto \cite{hyperfinite pinto}, Benci and Luperi Baglini \cite{benci_schwartz} and Bottazzi \cite{grid functions}.
Also, an earlier result by Todorov \cite{todorov delta} has some aspects in common with the non-Archimedean representation of the Dirac distribution by Flynn and Shamseddine.
Another approach to the non-Archimedean representation of distributions has been proposed by Oberbuggenberg and Todorov in \cite{todorov} and further studied in subsequent papers.
In this approach, the distributions are embedded in an algebra of asymptotic functions defined over a field of asymptotic numbers.
Some examples of the use of spaces of measurable functions of nonstandard analysis for the study of PDEs include works by Capi\'{n}sky and Cutland \cite{capicutland1,capicutland statistic}, Todorov \cite{todorov steady}, Benci and Luperi-Baglini \cite{ultrafunctions, ultramodel, ultraapps} and Bottazzi \cite{illposed}. For a more detailed survey of the use of nonstandard techniques for the study of PDEs, and in particular of classically ill-posed problems, we refer to the introduction of \cite{grid functions}.

The results obtained with techniques of Robinson's analysis with infinitesimals are more advanced than those obtained so far for the Levi-Civita field. This is due to the stronger topological properties of fields of hyperreal numbers and of the absence of a transfer principle for the Levi-Civita field, namely, of a procedure for extending sets and functions from $\R$ to $\civita$ in a way that elementary properties are preserved \cite{bottazzi}.
Despite these limitations of $\civita$, we believe it would be interesting to improve upon the results of \cite{shamflin2} and determine to what extent it is possible to develop a theory of generalized functions on the Levi-Civita field.
However, such a theory of generalized function would require a number of improvements upon the existing non-Archimedean measure theory.
For instance,
in order to justify the intuitive arguments that some distributions are the non-Archimedean counterparts of real distributions, we believe that it would be necessary to define an extension of functions from $\R$ to $\civita$ in a way that some relevant properties are preserved (as argued also for the development of a surreal calculus in \cite{integral}).

Currently, it is only possible to extend real continuous functions to locally analytic functions on the Levi-Civita field by means of the power series expansion at a point \cite{berz,calculusnumerics,analysislcf,Shamseddinephd}.
However, if the original function is not analytic, this extension does not preserve many first-order properties \cite{bottazzi}. As a consequence, it is well-known that many theorems of real analysis cannot be directly extended to the Levi-Civita field. For instance, in the Levi-Civita field the space of solutions of the differential equation $f'=0$ is infinite-dimensional, while in the real case it has dimension $1$ \cite{odes}.
The limitations of the extension of functions based upon the power serie expansion will become evident in Section \ref{section measurable functions}, where we will prove that the extension of a real function that is not locally analytic is not measurable. This result, combined with an earlier theorem of Shamseddine and Berz, entails that no measurable function provides a good representation of real functions that are not locally analytic at every point in their domain.

For this reason, 
we settle upon the representation of real measurable functions as sequences of measurable functions in $\civita$ that weakly converge in measure.
This representation allows to define a duality between measurable functions and representatives of continuous functions. This duality enables the definition of measurable functions that represent some real distributions, and the study of their properties. As a significant, but not exhaustive, example, we will discuss the well-known problem of the product between the Dirac and the Heaviside distribution, and we will show that the solution obtained in $\civita$ is consistent with the result obtained with other approaches based on Robinson's techniques, as discussed in \cite{grid functions}.
Moreover, we will show that, for these simple examples, the action of the pointwise derivative in the Levi-Civita field corresponds to that of the real distributional derivative. This result can be easily extended to other distributions that can be represented as measurable $C^k$ functions defined on $\civita$.

While our approach has improved upon some limitations of \cite{shamflin2}, it has also brought to light some drawbacks of the use of a $\civita$-valued measure in the definition of the non-Archimedean measurable functions.
For instance, it is not entirely possible to define a $C^\infty$ structure on the $\L^p$ spaces or on their completions, despite some partial progresses discussed in Section \ref{sec derivatives heaviside} (however, thanks to Corollary \ref{corollario derivate dirac} it would be possible to define a $C^k$ structure over the spaces $\L^p \cap C^k$).
Moreover, 
we are not confident that it would be possible to obtain an embedding of the space of distributions into the $\L^p$ or $\sL^p$ spaces (recall that it is possible to do so by working with hyperreal fields or with fields of asymptotic numbers). We will further remark on these limitations in Section \ref{sec conclusiva}.

Despite the limitations outlined above, we have chosen to work with the non-Archimedean measure and integration developed by Shamseddine, Berz and Flynn in order to establish a continuity with the theory developed so far for the Levi-Civita field.
However, in the light of the results emerged in our study, at the end of the paper we will suggest that the development of a theory of generalized functions in the Levi-Civita field might benefit from a radically different notion of measure. In particular, we believe that a real-valued measure over $\civita$, inspired by the construction of Loeb measures on hyperreal fields (presented for instance in \cite{loeb}), might allow for a better representation of real continuous functions and of their duality with measurable functions on the Levi-Civita field.
This idea will be explored in a subsequent paper \cite{forthcoming}.

Finally we remark that, despite the limitations just discussed, the results obtained in this paper might be relevant also from a broader perspective.
For instance, we are aware of an interest towards the Levi-Civita field as a field with infinitesimals whose definition does not depend upon principles such as the Axiom of Choice or the Ultrafilter Lemma \cite{pruss} (however, notice that we do not advocate the exclusive use of so-called effective methods in mathematics, as testified by our previous work on the hyperfinite representation of distributions and PDEs \cite{grid functions,illposed}. Our position is further clarified in \cite{gang,prussk,prusso,19c}).
Besides these alleged philosophical advantages of the Levi-Civita field, we think it is relevant to mention that, recently, Strobel used the Levi-Civita field for the implementation of algorithms of dynamic geometry based upon results obtained with techniques of nonstandard analysis \cite{strobel}.
Following an analogous idea, it might be possible that some measurable functions on the Levi-Civita field might be used for the implementation of hyperreal techniques for the representation of distributions or for the resolution of ODEs and PDEs.
Such algorithms could benefit from the results obtained in this paper, such as the representation of real continuous functions, the duality between measurable functions and representatives of real continuous functions, and the representation of the distributional derivative.
Moreover, the limitations of the measure and integration theory on $\civita$ might provide relevant theoretical restrictions for the scope of such implementations.

\subsection{Structure of the paper}
Section \ref{sec prelim} contains some basic notions on the Levi-Civita field.
In particular, we will review some properties of strong and weak convergence, and the basics of the theory of power series over $\civita$. We will also prove some technical properties that will be useful throughout the paper.

In Section \ref{section measure theory} we will review the non-Archimedean measure and integration on the Levi-Civita field and establish some novel results. For instance, we will provide many examples of nonmeasurable sets and we will show that it is not possible to approximate a measurable set with intervals whose measure is infinitely smaller than the measure of the set.

In our approach we chose also a slightly different characterization of simple and measurable functions. Namely, we will not require simple functions to be Lipschitz continuous and we will not assume that measurable functions are bounded. The first change will have no effect on the resulting measure theory, whereas the second will entail the existence of unbounded measurable functions. Despite this larger class of measurable functions, we will prove that
real functions that are not locally analytic at almost every point of their domain cannot be represented by measurable function on the Levi-Civita field.

The existence of unbounded measurable functions will be one of the motivations for the introduction of the $\L^p$ spaces in Section \ref{sec lp}. After having reviewed some of their basic properties, such as the existence of a $p$-norm that satisfies the familiar H\"older's inequality, we will determine that the $\L^p$ spaces are not sequentially complete with respect to strong and weak convergence in norm.
However, it is possible to define completions $\sL^p$ of the $\L^p$ spaces with respect to strong convergence. These completions turn out to be Banach spaces, and $\sL^2$ is also a Hilbert space.
However, it is not possible to represent real continuous functions even in these $\sL^p$ spaces.

For this reason, in Section \ref{section representation} we will settle for the representation of real continuous functions with weakly converging sequences of measurable functions.
However, it is not possible to define a completion of the $\L^p$ spaces with respect to weak convergence: this is mainly due the absence of a squeeze theorem for this notion of convergence.
Despite this negative result, the representation of real continuous functions with weakly converging sequences of measurable functions will allow us to define a class of measurable Delta and Heaviside functions on $\civita$ and to study their properties. In particular, we will discuss how in the Levi-Civita field it is possible to solve some simple problems from the nonlinear theory of distributions.

Finally, in Section \ref{sec conclusiva} we will present some concluding remarks and an outline for future research.

\subsection{Notation}

If $(\mathbb{F},<)$ is an ordered field and $a, b \in \mathbb{F}$, we will denote by $[a,b]_{\mathbb{F}}$ the set $\{ x \in \mathbb{F} : a \leq x \leq b \}$, and by $(a,b)_{\mathbb{F}}$ the set $\{ x \in \mathbb{F} : a < x < b \}$.
The sets $[a,b)_{\mathbb{F}}$ and $(a,b]_{\mathbb{F}}$ are defined accordingly.
The above definitions are extended in the usual way if $a = -\infty$ or $b = +\infty$, with the convention that $[a,+\infty]_{\mathbb{F}} = [a,+\infty)_{\mathbb{F}}$ and that $[-\infty,b]_{\mathbb{F}}=(-\infty,b]_{\mathbb{F}}$.
If $\mathbb{F}=\R$, we will often write $[a,b]$ instead of $[a,b]_{\R}$.

For all $a, b \in \civita$ with $a<b$, we will denote by $I(a,b)$ any of the intervals $(a,b)_\civita$, $[a,b)_\civita$, $(a,b]_\civita$ or $[a,b]_\civita$.
The length of an interval of the form $I(a,b)$ is denoted by $l(I(a,b))$ and is defined as $b-a$.

If $f: A\subseteq \R \rightarrow \R$, we define $\supp(f) = \overline{\{ x \in A : f(x) \ne 0 \}}$, where the closure is in the usual topology of $\R$.
If $f: A\subseteq \civita \rightarrow \civita$, we define $\supp(f) = \{ x \in \civita : f(x) \ne 0 \}$.

We will denote by $\an([a,b])$ the space of analytic functions over $[a,b] \subseteq \R$ and by $C^k([a,b])$ the space of $k$-times differentiable functions whose derivative of order $k$ is continuous over $[a,b]$.
If $f \in C^k([a,b])$, we will denote by $f^{(i)}$ its $i$-th derivative of order $0 \leq i \leq k$.
Notice that $f^{(0)} = f$.
We will sometimes write $f'$ instead of $f^{(1)}$ and $f''$ instead of $f^{(2)}$.

If $A \subseteq \R$ is a Lebesgue measurable set, we will denote by $\leb(A)$ its Lebesgue measure.

Consider the equivalence relation on Lebesgue measurable functions given by equality almost everywhere: two measurable functions $f$ and $g$ are equivalent if $\leb(\{x\in\Omega : f(x) \not = g(x)\})=0$.
We will not distinguish between the function $f$ and its equivalence class, and we will say that $f = g$ whenever the functions $f$ and $g$ are equal almost everywhere.

For all $1 \leq p < \infty$, $L^p(A)$ is the set of equivalence classes of measurable functions $f: A \rightarrow \R$ that satisfy
$
\int_{A} |f|^p dx < \infty.
$
If $f \in L^p(A)$, the $L^p$ norm of $f$ is defined by $\norm{f}_p^p = \int_{A} |f|^p dx.$

$L^\infty(A)$ is the set of equivalence classes of measurable functions that are essentially bounded: we will say that $f: A \rightarrow \R$ belongs to $L^\infty(A)$ if there exists $y \in \R$ such that $\leb(\{x\in\Omega:|f(x)| > y\}) = 0$.
In this case,
$$\norm{f}_\infty = \inf\{y \in \R : \leb(\{x\in A:f(x) > y\}) = 0 \}.$$

If $1 < p < \infty$, recall that $p'$ is defined as the unique solution to the equation
$
\frac{1}{p}+\frac{1}{p'}=1,
$
while $1' = \infty$ and $\infty' = 1$.

\section{Preliminary notions}\label{sec prelim}

The purpose of this section is to recall some results that will be used in the sequel of the paper.
We will briefly introduce the Levi-Civita field, the notion of weak limit, power series with coefficients in $\civita$, and the continuations of real functions to the Levi-Civita field.
The section includes also some novel results that will be used throughout the paper.

\subsection{The Levi-Civita field}

\begin{definition}
	A set $F \subset \Q$ is called left-finite if and only if for every $q \in \Q$ the set $\{x \in F : x \leq q \}$ is finite.
	The Levi-Civita field is the set $$\civita = \{ x:\Q \rightarrow \R: \{q:x(q)\not=0\} \text{ is left-finite} \},$$ together with the pointwise sum and the product defined by the formula
	$$
		(x \cdot y)(q) = \sum_{q_1+q_2=q} x(q_1) \cdot y(q_2).
	$$
\end{definition}

For a review of the algebraic and topological properties of $\civita$, we refer to \cite{berz, analysislcf, reexp, Shamseddinephd} and references therein.
Here we will only recall some fundamental properties and definitions of $\civita$ that will be useful in the sequel.

\begin{lemma}
	An element $x \in \civita$ is uniquely characterized by its support, which forms an ascending (possibly finite) sequence $\{q_n\}_{n \in \N}$, and a corresponding sequence $\{x[q_n]\}_{n\in\N}$ of function values.
\end{lemma}

\begin{definition}
	For all nonzero $x \in \civita$, define $\lambda(x) = \min (\text{supp}(x))$, and $\lambda(0) = \infty$.
	If $x, y \in \civita$, $x \not = y$, we say that $x < y$ if $(x-y)[\lambda(x-y)]<0$.
	The relation $x > y$ is defined as $y < x$, and the relations $\leq$ and $\geq$ are defined in the usual way.
\end{definition}

\begin{lemma}\label{lemma lambda valuation}
	The function $\lambda$ is a valuation with range $\Q \cup \{\infty\}$.
	In particular, for all $x, y \in \civita$, $\lambda(xy)=\lambda(x)+\lambda(y)$ and $\lambda(x+y)\geq \min\{\lambda(x),\lambda(y)\}$. Moreover, if $x\ne y$, $\lambda(x+y)> \min\{\lambda(x),\lambda(y)\}$.
\end{lemma}

The Levi-Civita field contains infinite and nonzero infinitesimal elements.
Moreover, the valuation $\lambda$ allows to compare the relative size of any element of $\civita$.

\begin{definition}
	Let $x, y \in \civita$, $x, y \geq 0$.
	We will write $x\ll y$ if $\lambda(x) > \lambda(y)$. In this case, we say that $x$ is infinitely smaller than $y$.
	This relation is extended in the expected way to all $x, y \in \civita$.
	Notice that $|x| \ll |y|$ if and only if $n|x| < |y|$ for all $n \in \N$.
\end{definition}

Comparison of the size of two elements in $\civita$ can be obtained not only via the inequalities $<$ and $\ll$, but also by introducing different notions of ``agreement on the order of magnitude'' and ``equality up to an infinitesimal relative error''.

\begin{definition}
	Let $x, y \in \civita$. We will write $x \sim y$ if $\lambda(x)=\lambda(y)$, and $x \approx y$ if $x\sim y$ and $x[\lambda(x)]=y[\lambda(y)]$.
\end{definition}

In \cite{calculusnumerics} it is observed that $\sim$ and $\approx$ are equivalence relations.

We will often reference the set
$$M_o = \{\varepsilon \in \civita : 0<\lambda(\varepsilon)<\infty\} = \{\varepsilon \in \civita : 0 < |\varepsilon| \ll1\},$$
that is the set of nonzero infinitesimal numbers in the Levi-Civita field.
In analogy with the common practice of nonstandard analysis, if $x \in \civita$, we will refer to the set $\mu(x)=\{ y \in \civita : |x-y|\ll1 \}$ as the monad of the point $x$.
Notice that monads are not intervals.

\begin{lemma}\label{monads are not intervals}
	For all $x, a, b \in \civita$ with $a<b$, $\mu(x) \not = I(a,b)$.
\end{lemma}
\begin{proof}
	If $x \not \in I(a,b)$, then the assertion is trivial.
	If $x \in I(a,b)$, suppose at first that $b-a \not \approx 0$: as a consequence, either $x \not \approx a$ or $x \not \approx b$. Since there exist $z, w \in I(a,b)$ with $z \approx a$ and $w \approx b$, we conclude $I(a,b)\not=\mu(x)$.
	If $b-a \approx 0$, then $x \approx a \approx b$, but we have also $x\approx a+2(b-a)\in \mu(x)$.
	However $a+2(b-a) \not \in I(a,b)$, so that $I(a,b)\subset \mu(x)$.
\end{proof}

We will also reference the number $d \in\civita$ as defined by Berz and Shamseddine \cite{analysislcf,berz+shamseddine2003,shamseddine2012} by the properties $d[1] = 1$ and $d[q] = 0$ for every $q \in \Q$, $q\ne1$.

Finally, we find it useful to borrow the definition of \emph{standard part} from Robinson's framework \cite{robinson} and define it also for elements of the Levi-Civita field.

\begin{definition}\label{def sh}
	Let $x \in \civita$.
	We define the function $\sh : \civita \rightarrow \R \cup \{\pm \infty\}$ by posing
	$$
		\sh{x} = 
		\left\{
		\begin{array}{ll}
		x[0] & \text{if } \lambda(x) \geq 0\\
		+\infty & \text{if } \lambda(x) < 0 \text{ and } x>0\\
		-\infty & \text{if } \lambda(x) < 0 \text{ and } x<0.
		\end{array}
		\right.
	$$
\end{definition}

\subsection{Convergence and weak convergence}

In the Levi-Civita field there are two notions of convergence. The first is induced by the metric and it is analogous to the usual definition of limit for real-valued sequences.
This notion of convergence is usually called strong convergence.

\begin{definition}\label{definition strong convergence}
	A sequence $\{a_n\}_{n \in \N}$ of elements of $\civita$ strongly converges to $l \in \civita$ if and only if
	$$
	\forall \varepsilon \in \civita, \varepsilon > 0, \exists n \in \N : \forall m > n\ |c_m-l|<\varepsilon.
	$$
	We will denote strong convergence with the expression $\slim_{n \rightarrow \infty} a_n = l$.
\end{definition}

Many properties of strong convergence have been established by Berz and Shamseddine \cite{berz,calculusnumerics,convergence}.
A property shared by convergence of real sequences and strong convergence in the Levi-Civita field is sequential completeness: in fact $\civita$ is sequentially complete with respect to strong convergence and every converging sequence in $\civita$ is bounded. Moreover, the squeeze theorem is still satisfied by strong convergence.

Throughout the paper we will use the following property: if a series with non-negative terms strongly converges to a number $l \in \civita$, then at least one of the terms has the same order of magnitude as $l$.
The proof of this result relies on the notion of regular sequence.

\begin{definition}
	A sequence $\{a_n\}_{n \in \N}$ of elements of $\civita$ is regular if and only if $\bigcup_{n \in \N}\supp(a_n)$ is a left-finite set.
\end{definition}

In \cite{calculusnumerics} it is proved that sequences that strongly converge in $\civita$ are regular.

We are now ready to prove the result stated above on series with non-negative terms.

\begin{lemma}\label{lemma tecnico serie a termini positivi}
	Let $\{a_n\}_{n\in\N}$ satisfy $a_n \geq 0$ for all $n \in \N$ and $\sum_{n \in \N} a_n = \slim_{k \rightarrow \infty} \sum_{n \leq k} a_n = l$.
	If $\lambda(l)=q$, then
	\begin{enumerate}
		\item $\lambda\left(a_n\right) \geq q$ for all $n \in \N$ and
		\item there exists $n \in \N$ such that $\lambda\left(a_n\right)=q$.
	\end{enumerate}
\end{lemma}
\begin{proof}
	In order to prove (1), suppose that $\lambda(a_k) < q$ for some $k \in \N$.
	Taking into account the hypothesis that $a_n \geq 0$ for all $n \in\N$, from Lemma \ref{lemma lambda valuation} we obtain $\lambda\left(\sum_{n \leq k} a_n \right)< q$.
	Since $a_n \geq 0$ for all $n \in\N$, the inequality $\sum_{n \leq k} a_n \leq \sum_{n \in \N} a_n$ entails also $\lambda\left(l\right)< q$.
	By contrapositive, we obtain the desired assertion.
	
	For the proof of (2), suppose towards a contradiction that $\max_{n \in \N} \lambda(a_n) = p > q$ for all $n \in \N$. 
	This maximum exists since strongly convergent sequences are regular.
	Let now $\varepsilon > 0$ satisfy $\lambda(\varepsilon)>p$, and let $k \in \N$ such that $\lambda \left(\sum_{n\leq k} a_n \right)= p$ and $\left|l-\sum_{n \leq k} a_n \right|<\varepsilon$.
	From the latter inequality we deduce that $l = \sum_{n \leq k} a_n + h$ for some $h \in \civita$ with $h \geq 0$ and $\lambda(h)\geq\lambda(\varepsilon)>p$.
	Since $a_n$, $l$ and $h$ are non-negative, we have $\lambda(l) = \min\left\{ \lambda(\sum_{n \leq k} a_n), \lambda(h)\right\} = p>q$, against the hypothesis $\lambda(l)=q$.
\end{proof}

Despite its good properties, strong convergence is very restrictive, since if $\{a_n\}_{n \in \N}$ is strongly convergent, then for all $q \in \Q$ the real sequence $n\mapsto a_n[q]$ is eventually constant \cite{analysislcf}.
As a consequence, real converging sequences are not strongly convergent in $\civita$, unless they are eventually constant.

In order to enlarge the class of converging sequences, it has been proposed a weaker notion of convergence, usually called weak convergence.

\begin{definition}
	A sequence $\{a_n\}_{n \in \N}$ of elements of $\civita$ weakly converges to $l \in \civita$ if and only if
	$$
	\forall \varepsilon \in \R\ \varepsilon > 0\ \exists n \in \N : \forall m > n\ \max_{q \in \Q, q\leq \varepsilon^{-1}} \left| (a_m - l)[q] \right|<\varepsilon.
	$$
	We will denote weak convergence with the expression $\wlim_{n \rightarrow \infty} a_n = l$.
\end{definition}

In \cite{berz} it is proved that strong convergence implies weak convergence, but the converse is false.

In order to determine whether a sequence is weakly convergent, it is convenient to use a weak converge criterion, stated for instance in Theorem 2.13 of \cite{convergence}.

\begin{theorem}\label{weak convergence criterion}
	If $\{a_n\}_{n\in\N}$ is a regular sequence of elements of $\civita$ and if there exists $l \in\civita$ such that for all $q \in\Q$
	$\lim_{n\rightarrow \infty} a_n[q] = l[q],$
	then $\wlim_{n\rightarrow \infty} a_n=l$.
	
	On the other hand, if $\{a_n\}_{n\in\N}$ converges weakly to $l \in \civita$, then for all $q \in \Q$ the real sequence $\{a_n[q]\}_{n\in\N}$ converges to $l[q]\in\R$, and the convergence is uniform on every subset of $\Q$ bounded above.
\end{theorem}

The above criterion allows to conclude that real converging sequences are weakly convergent also in $\civita$.

\begin{corollary}\label{corollario standard}
	If $\{r_n\}_{n\in\N}$ is a sequence of real numbers such that $\lim_{n \rightarrow \infty} r_n = r$, then $\wlim_{n\rightarrow\infty} r_n = r$.
\end{corollary}

Weak convergence does not satisfy some properties of convergence in $\R$ and of strong convergence in $\civita$.
One of them is the squeeze theorem.
In fact, if two sequences $\{a_n\}_{n \in \N}$ and $\{b_n\}_{n \in \N}$ satisfy
\begin{itemize}
	\item $a_i \leq b_j$ for all $i, j \in \N$ and
	\item $\wlim_{n\rightarrow \infty}a_n = l = \wlim_{n\rightarrow \infty} b_n$,
\end{itemize}
there might exist a sequence $\{c_n\}_{n \in \N}$ such that
$a_i \leq c_j \leq b_k$ for all $i, j, k \in \N$, but $\wlim_{n\rightarrow \infty} c_n$ might be different from $l$ or might not exist.

\begin{example}\label{example squeeze}
	Let $a_n = 0$ for all $n \in \N$ and let $b_n = n^{-1}$: then $\wlim_{n\rightarrow \infty} a_n = \wlim_{n \rightarrow \infty} b_n = 0$  by Corollary \ref{corollario standard}.
	Any constant sequence $\{c_n\}_{n\in\N}$ assuming a positive infinitesimal value $h$ satisfies $a_i \leq c_j \leq b_k$ for all $i, j, k \in \N$, however it converges to $h>0 = \wlim_{n\in\N} b_n$.
	On the other hand, if $c_n= nd$, then $\{c_n\}_{n \in \N}$ still satisfies $a_i< c_j < b_k$ for all $i, j, k \in \N$, but $\wlim_{n \rightarrow \infty} c_n$ does not exist.
\end{example}

The previous example shows also that weak limits do not preserve inequalities.
In particular, the weak limit of a strictly positive sequence might be negative.

\begin{example}\label{example convergence x norm}
	If $a_n = (nd)^{-1}-1$, then $a_n > 0$ for all $n \in \N$, but $\wlim_{n \rightarrow \infty} a_n = -1 < 0$.
\end{example}

In addition, weak convergence is not well-behaved with respect to the equivalence relation $\approx$.

\begin{remark}
	Let $\{a_n\}_{n\in\N}$ and $\{b_n\}_{n\in\N}$ be two sequences of elements of $\civita$.
	If $a_n \approx b_n$ for all $n \in\N$, it might be possible that $\{a_n\}_{n\in\N}$ converges (strongly or weakly), but $\{b_n\}_{n\in\N}$ does not.
	This happens for instance if $a_n = 1$ for all $n \in\N$ and $b_n = 1+nd$ for all $n \in\N$.
\end{remark}

Finally, whereas strongly converging sequences are regular, there are some weakly converging sequences that are not regular.

\begin{remark}
	In \cite{calculusnumerics}, it is shown that strongly converging sequences are regular. On the other hand there are sequences that are not regular but that weakly converge in $\civita$: an example is $\{n^{-1}d^{-n}\}_{n\in\N}$, whose weak limit is $0$.
	These sequences need not be unbounded: another weakly converging sequence that is not regular is $\{n^{-1}d^{-1+n^{-1}}\}_{n\in\N}$.
	
	As a consequence, there are some sequences $\{a_n\}_{n\in\N}$ that satisfy the formula
	\begin{equation}\label{formula weakly cauchy}
	\forall \varepsilon \in \R, \varepsilon > 0, \exists k \in \N : \forall m, n > k\ \max_{q \in \Q, q\leq \varepsilon^{-1}} \left| (a_m - a_n)[q] \right|<\varepsilon,
	\end{equation}
	that however do not converge weakly in $\civita$.
	An example is $\left\{\sum_{n \leq k} n^{-1}d^{-n} \right\}_{k\in\N}$, whose limit is the function $x: \Q \rightarrow \R$ defined by $x(-n)=n^{-1}$ whenever $n \in \N$, and $x(q)=0$ otherwise.
	This is not an element of $\civita$, since $\supp (x)$ is not a left-finite set.
\end{remark}

For a detailed study of the properties of the strong and weak convergence in the Levi-Civita field, we refer to \cite{berz,calculusnumerics, convergence} and references therein.

In the sequel, it will be useful to distinguish between sequences that are Cauchy with respect to the strong and weak convergence.

\begin{definition}
	A sequence $\{a_n\}_{n\in\N}$ of elements of $\civita$ is strongly Cauchy, or simply Cauchy, if and only if
	$$
	\forall \varepsilon \in \civita, \varepsilon > 0, \exists k \in \N : \forall m, n > k\ |a_m-a_n|<\varepsilon.
	$$
	A sequence $\{a_n\}_{n\in\N}$ is weakly Cauchy if and only if it satisfies formula \ref{formula weakly cauchy}.
\end{definition}

\subsection{Divergent sequences}

Besides the usual notions of strong and weak convergence, we find it convenient to introduce also a notion of divergent sequence.
In principle, in the Levi-Civita field this concept could be defined in different ways: for instance, it could be argued that all of the sequences $\{nd\}_{n\in\N}$, $\{n\}_{n\in\N}$ and $\{d^{-n}\}_{n\in\N}$ diverge in $\civita$.
It would be possible to define different notions of divergence that describe each of these behaviours; however, for our purposes we find it convenient to say that a sequence of elements of $\civita$ diverges if it is unbounded in the strong topology of $\civita$.

\begin{definition}\label{def diverge}
	A sequence $\{a_n\}_{n \in \N}$ of elements of $\civita$ diverges to $+\infty$ if and only if
	$$
		\forall M \in \civita\ \exists n \in \N\ \forall m > n\  a_m > M.
	$$
	
	A sequence $\{a_n\}_{n \in \N}$ of elements of $\civita$ diverges to $-\infty$ if and only if the sequence $\{-a_n\}_{n\in\N}$ diverges to $+\infty$
\end{definition}

Notice that, contrarily to the real case, monotonic sequences that do not diverge in $\civita$ might not have a strong or a weak limit, or even a converging subsequence.
In fact, if $\{a_n\}_{n \in \N}$ is an increasing sequence of elements of $\civita$, then it might exhibit any of the following behaviours:
\begin{enumerate}
		\item there exists $l \in \civita$ such that $\slim_{n \rightarrow \infty} a_n = l$;
		\item there exists $l \in \civita$ such that $\wlim_{n \rightarrow \infty} a_n = l$;
		\item the sequence diverges to $+\infty$ and converges weakly to some $l\in\civita$;
		\item the sequence diverges to $+\infty$ and does not converge weakly;
		\item the sequence is bounded from above and does not converge weakly or strongly.
\end{enumerate}

A divergent sequence $\{a_n\}_{n \in \N}$ such that $\wlim_{n \rightarrow \infty} a_n = 0$ is obtained by defining $a_n = n^{-1}d^{-n}$, and an increasing sequence $\{b_n\}_{n\in\N}$ that is bounded from above but does not have a weak or strong limit can be obtained by posing $b_n = nd$.

\subsection{Power series on the Levi-Civita field}\label{sec introcivita}

Real power series are series of the form $\sum_{n \in \N} a_n (x-x_0)^n$, where $x_0 \in \R$ and $x$ is a variable ranging in an interval $I$ (possibly of length $0$) centered at $x_0$; the number $l(I)/2$ is usually called the radius of convergence of the power series centred at $x_0$.

If $a_n \in \civita$ for all $n\in\N$ and if $x_0\in\civita$, it is possible to define a power series either as the strong limit or the weak limit of the sequence of the partial sums $\sum_{n \leq k} a_n(x-x_0)^n$: 
\begin{equation}\label{kinds of limit}
	\sum_{n \in \N} a_n (x-x_0)^n = \slim_{k \rightarrow \infty} \sum_{n \leq k} a_n(x-x_0)^n
	\text{ or }
	\sum_{n \in \N} a_n (x-x_0)^n = \wlim_{k \rightarrow \infty} \sum_{n \leq k} a_n(x-x_0)^n.
\end{equation}
Strong and weak convergence of power series has been studied in \cite{calculusnumerics, reexp, analytic1, convergence}, where there have been proved some convergence criteria for both the strong and weak limit of equation \eqref{kinds of limit}.
It has been argued by Berz and Shamseddine that the most convenient definition of power series in the Levi-Civita field is the one relying on the weak limit. With this definition, power series share many properties with their real counterpart, such as the possibility to calculate the derivative by differentiating
the power series term by term, or the possibility to re-expanded the power series around any point in its domain of convergence.
Other properties of the power series defined using the weak limit are discussed in detail in \cite{reexp,convergence}.
From now on, the expression $\sum_{n \in \N} a_n (x-x_0)^n$ will denote the weak limit $\wlim_{k \rightarrow \infty} \sum_{n \leq k} a_n (x-x_0)^n$.

\begin{definition}
	For every $A\subseteq\civita$, we will denote by $\P(A)$ the algebra  of power series that weakly converge for every $x\in A$.
\end{definition}

The functions defined as power series with coefficients in $\civita$ satisfy many properties of real continuous functions, such as an intermediate value theorem, an extreme value theorem and a mean value theorem.
For these results, we refer to \cite{reexp,analytic1,convergence}.

There is a strong relation between real power series and power series in $\civita$ with real coefficients: if $a_n \in \R$ for all $n\in\N$ and if the real power series $\sum_{n \in \N} a_n (x-x_0)^n$ has radius of convergence $R$, the power series $\sum_{n \in \N} a_n (x-x_0)^n$ converges weakly in $\civita$ whenever $|x-x_0| < R$ and $|x-x_0| \not \sim R$. As a consequence, in the region of convergence it is possible to define a function $x \mapsto \wlim_{k \rightarrow \infty}\sum_{n \leq k} a_n (x-x_0)^k$; moreover, whenever $x$ is real, this function agrees with the real power series $\sum_{n \in \N} a_n (x-x_0)^n$ (see Definition \ref{def ext} for more details).

If $|x-x_0| < R$ and $|x-x_0| \sim R$, however, the power series in the Levi-Civita field might not converge. This is a consequence of the following result, that does not require the additional hypothesis $a_n \in \R$ for all $n \in\N$.

\begin{proposition}\label{lemma divergenza}
	Let $\{a_n\}_{n\in\N}$ be a sequence in $\civita$ and let $x_0\in\civita$.
	\begin{enumerate}
		\item If for some $x \in \civita$ the power series $\sum_{n \in \N} a_n (x-x_0)^n$ diverges, namely if the sequence $k \mapsto \sum_{n \leq k} a_n (x-x_0)^n$ diverges in the sense of Definition \ref{def diverge}, then for every $h \in \civita$ with $\lambda(h)>\lambda(x-x_0)$, the power series $\sum_{n \in \N} a_n (x+h-x_0)^n$ diverges.
		\item If for some $x \in \civita$ the power series $\sum_{n \in \N} a_n (x-x_0)^n$ does not converge, namely if the sequence $k \mapsto \sum_{n \leq k} a_n (x-x_0)^n$ does not converge weakly in the sense of Definition \ref{definition strong convergence}, there exists $s\in\Q$ such that for every $h$ with $\lambda(h)>s$, the power series $\sum_{n \in \N} a_n (x+h-x_0)^n$ does not converge.
	\end{enumerate}
\end{proposition}
\begin{proof}
	A consequence of Theorem 3.8 of \cite{analysislcf} is that $\sum_{n \in \N} a_n (x-x_0)^n$ converges weakly if and only if $\sum_{n \in \N} |a_n (x-x_0)^n|$ converges weakly.
	Thus we can assume without loss of generality that $a_n \geq 0$ for all $n \in \N$ and that $x-x_0 > 0$.
	Notice also that $\lambda(x-x_0)< \infty$, since any power series strongly converges whenever $x=x_0$.
	
	For every $h$ satisfying the hypothesis $\lambda(h)> \lambda(x-x_0)$, we have the equality $\lambda(x+h-x_0)=\lambda(x-x_0)$.
	From this equality and from Lemma \ref{lemma lambda valuation} we obtain also
	\begin{equation}\label{eq prop divergenza-non convergenza}
	\lambda[(x+h-x_0)^n]=\lambda[(x-x_0)^n]=n\lambda[x-x_0]
	\end{equation}
	and
	\begin{equation}\label{eq approx serie}
	\sum_{n \leq k} a_n (x-x_0)^n \approx \sum_{n \leq k} a_n (x+h-x_0)^n.
	\end{equation}
	for all $n \in\N$.

	(1) Since $\sum_{n \in \N} a_n (x-x_0)^n$ diverges, for all $q\in\Q$ there exists $k_q \in \N$ such that $\sum_{n \leq k_q} a_n (x-x_0)^n > d^{q}$.
	Thanks to equation \eqref{eq approx serie}, we have also the inequality $\sum_{n \leq k_q} a_n (x+h-x_0)^n > d^{q}$.
	Thus if $M\in\civita$ and if $k>k_{\lambda(M)}$, then $\sum_{n \leq k} a_n (x-x_0)^n > M$, so that the series $\sum_{n \in \N} a_n (x+h-x_0)^n$ diverges.

	(2)	Suppose now that $\sum_{n \in \N} a_n (x-x_0)^n$ does not converge.
	If $\sum_{n \in \N} a_n (x-x_0)^n$ diverges, then by part (1) of the proof we can choose $s=\lambda(x-x_0)$.
	
	Under the hypothesis that $\sum_{n \in \N} a_n (x-x_0)^n$ does not diverge, there exists $p\in\Q$ such that $\lambda\left(\sum_{n \leq k} a_n (x-x_0)^n\right) \geq p$ for all $k \in\N$.
	Thanks to the first part of the weak convergence criterion of Theorem \ref{weak convergence criterion}, then either the sequence $\left\{\sum_{n \leq k} a_n (x-x_0)^n\right\}_{k\in\N}$ is not regular or there exists $q \in \Q$ such that the real sequence $\left\{\left(\sum_{n \leq k} a_n (x-x_0)^n\right)[q]\right\}_{k\in\N}$ does not converge.
	
	At first, suppose that $\left\{\sum_{n \leq k} a_n (x-x_0)^n\right\}_{k\in\N}$ is regular and chose $q \in \Q$ in a way that the real sequence $\left\{\left(\sum_{n \leq k} a_n (x-x_0)^n\right)[q]\right\}_{k\in\N}$ does not converge.
	Notice that, by definition of $p$, $q\geq p$.
	Recall the formula for the difference of two $n$-th powers:
	$$
		(x+h-x_0)^{n}-(x-x_0)^{n} = h \sum_{j = 0}^{n-1}(x+h-x_0)^{n-j-1}(x-x_0)^{j}.
	$$
	Under the hypothesis that $\lambda(h)>\lambda(x-x_0)$, the previous equality and equation \eqref{eq prop divergenza-non convergenza} entail that for all $k\in\N$
	\begin{equation}\label{eqn proposizione divergenza}
	\lambda\left(\sum_{n \leq k} a_n (x+h-x_0)^n-\sum_{n \leq k} a_n (x-x_0)^n\right) = \lambda(h)+\lambda\left(\sum_{n \leq k} a_n (x-x_0)^n\right) \geq \lambda(h)+p.
	\end{equation}
	Hence, if $\lambda(h)>\lambda(x-x_0)$ and $\lambda(h) + p > q$, then also
	$$
	\left(\sum_{n \leq k} a_n (x+h-x_0)^n\right)[q]
	=
	\left(\sum_{n \leq k} a_n (x-x_0)^n\right)[q],
	$$
	so that the real sequence $\left\{\left(\sum_{n \leq k} a_n (x+h-x_0)^n\right)[q]\right\}_{k\in\N}$ does not converge.
	The second part of the weak convergence criterion of Theorem \ref{weak convergence criterion} ensures that this is sufficient to entail that the sequence $\left\{\sum_{n \leq k} a_n (x+h-x_0)^n\right\}_{k\in\N}$ cannot converge weakly.
	We obtain the desired assertion by defining $s = \max\{\lambda(x-x_0), q-p\}$.
	
	Suppose now that the sequence $\left\{\sum_{n \leq k} a_n (x-x_0)^n\right\}_{k\in\N}$ is not regular.
	Then there exists $q \in \Q$ such that
	$$
		\bigcup_{k\in\N} \supp\left(\sum_{n \leq k} a_n (x-x_0)^n\right) \cap (-\infty,q]_{\Q}
	$$
	is not finite.
	If $\lambda(h)>\lambda(x-x_0)$ and $\lambda(h) > q-p$, then by equation \eqref{eqn proposizione divergenza}
	$$
		\bigcup_{k\in\N} \supp\left(\sum_{n \leq k} a_n (x-x_0)^n\right) \cap (-\infty,q]_{\Q}
		=
		\bigcup_{k\in\N} \supp\left(\sum_{n \leq k} a_n (x+h-x_0)^n\right) \cap (-\infty,q]_{\Q},
	$$
	so that also the sequence $\left\{\sum_{n \leq k} a_n (x+h-x_0)^n\right\}_{k\in\N}$ is not regular.
	Once again, we can define $s = \max\{\lambda(x-x_0), q-p\}$.
\end{proof}

If the coefficients of the power series are real
, then we can sharpen the above results.

\begin{corollary}
	Consider the power series $\sum_{n \in \N} a_n (x-x_0)^n$ with $x_0 \in\R$ and $a_n\in \R$ for all $n\in\N$.
	Then the power series $\sum_{n \in \N} a_n (x-x_0)^n$ does not diverge for any $x \in \civita$ such that $\lambda(x) \geq 0$.
	Moreover, if for some $x \in\R$ the series $\sum_{n \in \N} a_n (x-x_0)^n$ does not converge weakly in $\civita$, then for all $h \in M_o$ the series $\sum_{n \in \N} a_n (x+h-x_0)^n$ does not converge weakly.
\end{corollary}
\begin{proof}
	If $a_n = 0$ for all but finitely many $n \in \N$, then both assertions are trivial. Suppose then that $a_n \ne 0$ for infinitely many $n \in \N$.
	
	The hypotheses over $a_n$, $x_0$ and $x$ entail that
	\begin{enumerate}
		\item $\lambda(a_n(x-x_0)^n)=0$ or $\lambda(a_n(x-x_0)^n)=\infty$ for all $n \in\N$, and
		\item $\lambda\left(\sum_{n \leq k} a_n (x-x_0)^n\right)\geq0$ for all $k\in\N$.
	\end{enumerate}
	Then the power series does not diverge as a consequence of inequality (2).
	
	If $x\in\R$ and the series $\sum_{n \in \N} a_n (x-x_0)^n$ does not converge weakly, then
	from our hypotheses over $a_n$, from equation \eqref{eqn proposizione divergenza} with $p=0$ and from (1) we obtain that, whenever $\lambda(h)>0$,
	$$
	\left(\sum_{n \leq k} a_n (x+h-x_0)^n\right)[0]
	=
	\left(\sum_{n \leq k} a_n (x-x_0)^n\right)[0].
	$$
	Consequently, if $h\in M_o$ then also the series $\sum_{n \in \N} a_n (x+h-x_0)^n$ does not converge.
\end{proof}

As a consequence of this corollary, there are some real power series centred at some real point $x_0$ that converge in an open interval $(x_0-R,x_0+R)$, but whose counterpart in $\civita$ do not converge in any interval of the form $I(x_0-R,x_0+R)$.

\begin{example}
	Consider the power series $\sum_{n \in \N} x^n$: it has real coefficients, so whenever $|x|<1$ and $|x|\not\approx 1$ it converges to a function that can be considered as an extension to the Levi-Civita field of the real function $f(x)=(1-x)^{-1}$ (for a more precise statement of this result, we refer to Definition \ref{def ext} and to \cite{bottazzi}).
	However, if $x \approx 1$, the series does not converge, since $\sum_{n \leq k}x^n \approx \sum_{n \leq k} 1 = k$.
	Similarly, if $x \approx -1$, the series does not converge, since $\sum_{n \leq k}x^n \approx \sum_{n \leq k} -1^n$.
	
	Now consider the function $f : \civita \setminus \{0\} \rightarrow \civita$ defined by $f(x) = x^{-1}$. 
	The power series expansion of $f$ around any $x_0 > 0$ is
	$f(x_0-x)= \sum_{n\in\N} x_0^{-n-1} x^n$; this series weakly converges iff $\frac{|x|}{x_0} < 1$ and $\frac{|x|}{x_0} \not\approx 1$.
\end{example}

The behaviour described in the first part of the previous example can be expressed with an analogy from Robinson's framework of analysis with infinitesimals.

\begin{remark}
	The behaviour of power series in the Levi-Civita field can be described by introducing the notion of \emph{nearstandard point}, commonly used in the topology of hyperreal fields.
	We will say that a point $x \in A \subseteq \civita$ is nearstandard in $A$ iff $\sh{x} \in A$.
	
	With this terminology, we can say that if $\sum_{n \in \N} a_n (x-x_0)^n$ is a real power series with radius of convergence $R$, then its Levi-Civita counterpart weakly converges on the nearstandard points of $(x_0-R,x_0+R)_\civita$, but it might not converge on those points that are not nearstandard in $(x_0-R,x_0+R)_\civita$.
\end{remark}

\subsection{Extending real functions to the Levi-Civita field}\label{section extension}

Since power series with real coefficients weakly converge also in $\civita$, it is possible to use them to define several extensions of real continuous functions to the Levi-Civita field.
These extensions are obtained from the Taylor series expansion of a function at a point.

\begin{definition}\label{def ext}
	Let $f \in C^\infty([a,b])$.
	The analytic extension of $f$ is defined as
	$$
	\ext{\infty}{f}(r+\varepsilon) = \sum_{i = 0}^{\infty} f^i(r) \frac{\varepsilon^i}{i!}.
	$$
	for all $r \in [a,b]$ and for all $\varepsilon \in M_o$ such that $r+\varepsilon \in [a,b]_{\civita}$.
	If $f$ is analytic, then $\ext{\infty}{f}$ will be called the canonical extension of $f$.
	The only exceptions are the canonical extensions of the exponential function and of the trigonometric functions sine and cosine, still denoted by $e^x$, $\sin(x)$ and $\cos(x)$. Notice also that these extensions are defined for all $x \in \civita$ with $\lambda(x)\geq0$.
	
	Let $f \in C^k([a,b])$, possibly with $k = \infty$.
	The order $n$ extension of $f$, with $0\leq n\leq k$, is defined by
	$$
	\ext{n}{f}(r+\varepsilon) = \sum_{i = 0}^{n} f^i(r) \frac{\varepsilon^i}{i!}.
	$$
	for all $r \in [a,b]$ and for all $\varepsilon \in M_o$ such that $r+\varepsilon \in [a,b]_{\civita}$.
\end{definition}

\begin{remark}
	Definition \ref{def ext} applies also to functions whose domains are intervals of the form $(a,b)$, $[a,b)$, $(a,b]$: the only difference is that the extensions of $f$ will be defined over the nearstandard points of the corresponding interval $(a,b)_\civita$, $[a,b)_\civita$, $(a,b]_\civita$.
\end{remark}

It is possible to prove that each of the extensions introduced in Definition \ref{def ext} is unique and well-defined.
The extensions $\ext{\infty}{f}$ and $\ext{n}{f}$ extend the corresponding real function $f\in C^k([a,b])$ in the sense that for every $x \in [a,b]$ $f(x)=\ext{n}{f}$ for all $n \leq k$, possibly with $k = \infty$.
For more properties of the continuations of real functions to the Levi-Civita field, we refer to \cite{berz,analysislcf,reexp,computational,bottazzi,Shamseddinephd} and references therein.

\subsection{On the weak pointwise convergence of sequences of analytic functions}

In the Levi-Civita field, weak pointwise convergence of a sequence of canonical extensions of analytic functions is possible only under very strong hypotheses.

\begin{proposition}\label{prop convergence}
	Let $\{f_n\}_{n \in \N}$ be a sequence of analytic functions over $[a,b]$ and let $g_n=\ext{\infty}{\left(f_n\right)}$ for all $n\in\N$.
	Then for all $x \in [a,b]_{\civita}$ the sequence $\left\{g_n(x)\right\}_{n \in \N}$ converges weakly for all $x \in [a,b]_\civita$ if and only if for all $r \in [a,b]$ and for all $p \in \Q$, $p \geq 0$,
	\begin{equation}\label{sums that must converge}
	\lim_{n\rightarrow \infty}\ \sum_{i\leq p} f_n^{(i)}(r)
	\end{equation}
	exists and is finite.
\end{proposition}
\begin{proof}
	Let $x=r+h\in[a,b]_\civita$ with $r \in \R$ and $h \in M_o$.	
	For every $q \in \Q$,
	\begin{equation}\label{equality g[q] and sums of derivatives of f}
		g_n(x)[q] = \sum_{i\leq \frac{q}{\lambda(h)}} f_n^{(i)}(r) \frac{h^{i}[q]}{i!}.
	\end{equation}
	Notice also that, thanks to the hypotheses over $f_n$ and $x$, $g_n(x)[q] = 0$ whenever $q<0$.
	
	Assume that the sequence
	$
	n \mapsto \sum_{i\leq \frac{q}{\lambda(h)}} f_n^{(i)}(r)
	$
	has a finite limit for all $q \in \Q$, $q \geq 0$.
	Since $\lim_{i \rightarrow \infty} \frac{h^{i}[q]}{i!}=0$ for all $q \in \Q$, the real sequence $\left\{\frac{h^{i}[q]}{i!}\right\}_{i \in \N}$ is bounded from above by some $b_q \in \R$.
	Thus,
	for all $q\in\Q$
	$$
	\left| \sum_{i\leq \frac{q}{\lambda(h)}} f_n^{(i)}(r) \frac{h^{i}[q]}{i!} \right|
	\leq |b_q| \left| \sum_{i\leq \frac{q}{\lambda(h)}} f_n^{(i)}(r) \right|.
	$$
	By our hypothesis over the sequences \ref{sums that must converge}, we conclude that
	also
	$n \mapsto \sum_{i\leq \frac{q}{\lambda(h)}} f_n^{(i)}(r) \frac{h^{i}[q]}{i!}$
	has a finite limit as $n\rightarrow\infty$.
	As a consequence, we can define a number $g(x) \in \civita$ by posing
	$$
	g(x)[q] = \lim_{n\rightarrow \infty} \sum_{i\leq \frac{q}{\lambda(h)}} f_n^{(i)}(r) \frac{h^{i}[q]}{i!}.
	$$
	By Theorem \ref{weak convergence criterion}, the sequence $\{g_n(x)\}_{n\in\N}$ converges weakly to $g(x)$.
	
	Now suppose that there exists $q \in \Q$, $q \geq 0$, such that \eqref{sums that must converge} does not exist or is not finite.
	Thanks to equality \eqref{equality g[q] and sums of derivatives of f}, the real sequence $g_n(x)[q]$ does not converge.
	By Theorem \ref{weak convergence criterion}, the sequence $\{g_n(x)\}_{n\in\N}$ does not converge weakly for some $x \in [a,b]_\civita$.
\end{proof}

The condition expressed in Theorem \ref{prop convergence} is rarely satisfied even by sequences of functions that converge uniformly in $\R$, as it is shown in the next example.

\begin{example}\label{example sin/n}
	Consider the sequence of analytic functions $\{f_n\}_{n \in \N}$, with $f_n(x)=\frac{\sin(n^2x)}{n}$.
	This sequence converges uniformly to $0$
	for all $x \in \R$, but $\lim_{n \rightarrow \infty} f'_n(x) = \lim_{n \rightarrow \infty} n\cos(n^2x)$ does not converge for any $x \in \R$.
	
	Consider now the sequence of the canonical extensions $\{g_n\}_{n \in \N}$ defined as in Theorem \ref{prop convergence}.
	Thanks to Corollary \ref{corollario standard}, $\{g_n(r)\}_{n\in\N}$ converges weakly to $0$ whenever $r \in \R$.
	However, ${g_n}(d)=n\cos(n^2d)$: since $\cos(n^2d)\approx 1$ for all $n \in\N$, we deduce that $\{{g_n}(d)\}_{n\in\N}$ does not converge weakly.
	As a consequence, it is not possible to define a weak pointwise limit for the
	canonical extension of the uniformly convergent sequence $\{f_n\}_{n \in \N}$.
\end{example}

\section{A uniform measure on the Levi-Civita field}\label{section measure theory}

A uniform measure for the Levi-Civita field has been developed by Shamseddine and Berz in \cite{shamseddine2012,berz+shamseddine2003} and extended to $\civita^2$ and $\civita^3$ by Shamseddine and Flynn \cite{shamflin2, shamflin1}.
The underlying idea is to define a Lebesgue-like measure over subsets of $\civita$ starting from the notion of length of an interval. 
Once the measurable sets are defined, it is possible to introduce a family of simple functions and finally the space of measurable functions. 
Besides the main definitions and results from \cite{shamseddine2012,berz+shamseddine2003}, we will also present some novel results and remarks that will be relevant for the sequel of the paper.

The only differences between our approach and the one proposed by Shamseddine and Berz will be in the definition of simple functions and in the definition of measurable functions.
The former are not required to be Lipschitz continuous and are defined over arbitrary intervals, instead of being Lipschitz and defined over closed intervals.
Thanks to Proposition \ref{lemma divergenza}, this definition turns out to be equivalent to the one proposed by Shamseddine and Berz.

In addition, measurable functions are not required to be bounded. As a consequence of avoiding this hypothesis, we will obtain a wider space of measurable functions.
The richer space will provide a motivation for the introduction of the $L^p$ spaces over the Levi-Civita field, as discussed in Remark \ref{remark if bounded then trivial}.

\subsection{Measurable sets}

In analogy with the Lebesgue measure, a set is measurable in $\civita$ if it can be approximated with arbitrary precision by a countable sequence of intervals.

\begin{definition}\label{measurable sets}
	A set $A \subseteq \civita$ is measurable iff for every $\varepsilon\in\civita$ there exist two sequences of mutually disjoint intervals $\{I_n\}_{n\in\N}$ and $\{J_n\}_{n\in\N}$ such that
	\begin{enumerate}
		\item $\bigcup_{n\in\N} I_n \subseteq A \subseteq \bigcup_{n\in\N} J_n$;
		\item $\sum_{n\in\N} l(I_n)$ and $\sum_{n\in\N} l(J_n)$ strongly converge in $\civita$;
		\item $\sum_{n\in\N} l(J_n) - \sum_{n\in\N} l(I_n) \leq \varepsilon$.
	\end{enumerate}
\end{definition}

By exploiting condition (3) of Definition \ref{measurable sets}, it is possible to define a measure for any measurable set $A$.

\begin{definition}\label{definizione misura}
	If $A\subset\civita$ is a measurable set, then for every $k \in \N$ there exist two sequences of mutually disjoint intervals $\left\{I^k_n\right\}_{n \in \N}$ and $\left\{J^k_n\right\}_{n \in \N}$ satisfying properties (1)-(2) of Definition \ref{measurable sets} together with the inclusions
	$$
		\bigcup_{n \in \N} I^k_n \subseteq \bigcup_{n \in \N} I^{k+1}_n
		\subseteq A \subseteq
		\bigcup_{n \in \N} J^{k+1}_n \subseteq \bigcup_{n \in \N} I^k_n  
	$$
	and the inequality (3) with $\varepsilon = d^k$:
	$$
		\sum_{n\in\N} l(J^k_n) - \sum_{n\in\N} l(I^k_n) \leq d^k.
	$$
	The measure of $A$, denoted by $\m(A)$, is defined as
	$$
		\m(A)=\slim_{k \rightarrow \infty} \sum_{n\in\N} l(I^k_n) = \slim_{k \rightarrow \infty} \sum_{n\in\N} l(J^k_n).
	$$
\end{definition}

In \cite{berz+shamseddine2003} it is proved that $\m(A)$ is well-defined and that
\begin{eqnarray*}
	\m(A) & = & \sup \left\{ \sum_{n\in\N} l(I_n) : \{I_n\}_{n \in \N} \text{ is a sequence of mutually disjoint intervals with } \bigcup_{n \in \N} I_n \subseteq A \right\} \\
	& = & \inf \left\{ \sum_{n\in\N} l(J_n) : \{J_n\}_{n \in \N} \text{ is a sequence of mutually disjoint intervals with } A \subseteq \bigcup_{n \in \N} J_n \right\}.
\end{eqnarray*}

\begin{remark}
According to Definition \ref{definizione misura}, measurable sets have bounded measure, i.e.\ there does not exist a measurable set $A$ with $\m(A)=+\infty$.
As a consequence, unbounded intervals of the form $I(-\infty, b)=\{x \in \civita: x < b\}$, $I(a,+\infty)=\{x \in \civita: a<x \}$ and even $\civita$ itself are not measurable. We believe it would be possible to extend the definition of measure in order to include these unbounded intervals among the measurable sets, but in this paper we do not pursue this idea.
\end{remark}

Due to the presence of infinitesimal elements in $\civita$, the resulting measure turns out to be rather different from the Lebesgue measure over $\R$.

The properties of measurable sets are studied in detail in \cite{shamseddine2012,berz+shamseddine2003}.
We recall some useful properties: if $A$ and $B$ are measurable, then also $A \cup B$ and $A \cap B$ are measurable, moreover $\m(A \cup B)= \m(A)+\m(B)-\m(A\cap B)$.
In addition, $\m$ is complete in the sense that if $A$ is measurable and $\m(A) =0$, then every set $B \subseteq A$ is measurable and $\m(B) = 0$.
Indeed, this result can be obtained as a particular case of the following property of measurable sets.

\begin{lemma}\label{lemma A c B c C}
	Let $B \subseteq A \subset \civita$ be measurable sets.
	If $\m(A) = \m(B)$, then for every $C \subset \civita$ that satisfies $B \subseteq C \subseteq A$, $C$ is measurable and $\m(C) = \m(A) = \m(B)$.
\end{lemma}
\begin{proof}
	Let $A, B$ and $C$ satisfy the hypotheses of the lemma.
	Since $\m(A) = \m(B)$, for every $k\in\N$ there exist two sequences of mutually disjoint intervals $\left\{I^k_n\right\}_{n \in \N}$ and $\left\{J^k_n\right\}_{n \in \N}$ satisfying
	\begin{itemize}
		\item $\bigcup_{n \in \N} I^k_n \subseteq B \subseteq C \subseteq A \subseteq \bigcup_{n \in \N} J^k_n$;
		\item $\sum_{n\in\N} l(I_n)$ and $\sum_{n\in\N} l(J_n)$ strongly converge in $\civita$;
		\item since $\m(A) = \m(B)$, $\sum_{n\in\N} l(J_n) - \sum_{n\in\N} l(I_n) \leq d^{-k}$.
	\end{itemize}
	From these properties, we deduce that $\left\{I^k_n\right\}_{n \in \N}$ and $\left\{J^k_n\right\}_{n \in \N}$ satisfy Definition \ref{measurable sets} for $C$, with $\varepsilon = d^{-k}$. Thus $C$ is measurable. From Definition \ref{definizione misura}, we deduce that $\m(C) = \m(A) = \m(B)$, as desired.
\end{proof}

Some of the usual theorems of the Lebesgue measure for $\R^k$ are not satisfied by the uniform measure over $\civita$: for instance, the complement of a measurable set is not necessarily measurable.
An example is given by $\Q\cap [0,1]_{\civita}$, that is measurable and has measure $0$, and $[0,1]_{\civita}\setminus \Q$, that is not measurable, since it only contains intervals of an infinitesimal length.
As a consequence, the family of measurable sets is not even an algebra.

This example is not isolated, and in fact in the Levi-Civita field there are many families of sets that are not measurable.
An example is given by the monad at each point: this is a special case of the following result.

\begin{lemma}\label{lemma monad}
	For every $x \in \civita$ and for every $q \in \Q$, the sets
	\begin{enumerate}
		\item $\{y \in \civita : \lambda(x-y)>q\}$;
		\item $\{y \in \civita : \lambda(x-y)\geq q\}$;
		\item $\{y \in \civita : \lambda(x-y)<q\}$;
		\item $\{y \in \civita : \lambda(x-y)\leq q\}$;
		\item $\{y \in \civita : \lambda(x-y)= q\}$;
	\end{enumerate}
	are not measurable.
\end{lemma}
\begin{proof}
(1) For all $x \in \civita$ and for all $q\in\Q$, if $I(a,b)$ and $J(c,d)$ are two intervals such that $I \subseteq \{y \in \civita : \lambda(x-y)>q\} \subseteq J$, then $\lambda(d-c)<\lambda(b-a)$: as a consequence, property (3) of Definition \ref{measurable sets} cannot be satisfied.
A similar proof applies also to sets of the form (2), (3), (4) and (5).
\end{proof}

In contrast to what happens with the Lebesgue measure, countable unions of measurable sets might be non measurable.

\begin{lemma}\label{lemma countable union}
	If $\{K_n\}_{n \in\N}$ is a family of mutually disjoint sets such that $\sum_{n \in \N} l(K_n)$ does not strongly converge, then $A=\bigcup_{n \in \N} K_n$ is not measurable.
\end{lemma}
\begin{proof}
	Let $\{I_n\}_{n \in \N}$ be a sequence of mutually disjoint intervals such that $\sum_{n \in \N} l(I_n)$ strongly converges and $\bigcup_{n \in \N} I_n \subseteq A$.
	By strong convergence of $\sum_{n \in \N} l(I_n)$, there exists $l \in \civita$ such that for every $\varepsilon\in\civita$, $\varepsilon>0$ there exists $k_\varepsilon \in \N$ satisfying $l - \sum_{n \leq k} l(I_n) < \varepsilon$ for every $k > k_\varepsilon$.
	Since $\sum_{n \in \N} l(K_n)$ does not strongly converge and since $l(K_n) \geq 0$ for all $n \in \N$, there exists $\varepsilon\in\civita$, $\varepsilon>0$ and $j \in \N$ such that $\left| \sum_{n \leq k} l(K_n) - l \right| > \varepsilon$ for every $k > j$.
	Moreover, since $\bigcup_{n \in \N} I_n \subseteq \bigcup_{n \in \N} K_n$, the above inequality can be sharpened to $\sum_{n \leq k} l(K_n)  - l > \varepsilon$ for every $k > j$.
	
	Similarly, for every sequence of mutually disjoint intervals $\{J_n\}_{n \in \N}$ such that $\sum_{n \in \N} l(J_n)$ strongly converges and $A \subseteq \bigcup_{n \in \N} J_n$, by monotonicity  of the measure $\m$ we have the inequality $\sum_{n \leq k} l(K_n) < \sum_{n \in \N} l(J_n)$ for all $k \in \N$ .
	
	Putting the various inequalities together, if $k > \max\{k_\varepsilon,j\}$ then
	$$
		\sum_{n \in \N} l(I_n) < \sum_{n \leq k} l(K_n) < \sum_{n \in \N} J_n.
	$$
	
	We deduce that every sequence of mutually disjoint intervals $\{I_n\}_{n \in \N}$ and $\{J_n\}_{n \in \N}$ such that $\bigcup_{n \in \N} I_n \subseteq A \subseteq \bigcup_{n \in \N} J_n$ cannot satisfy condition (3) of Definition \ref{measurable sets}.
\end{proof}

The hypotheses of the above Lemma are satisfied for instance if $l(K_n)=l\ne0$ for all $n \in \N$.
For other examples of measurable and non measurable sets and for a more detailed development of the measure theory on $\civita$, we refer to \cite{moreno,shamseddine2012,berz+shamseddine2003}.

As in the real measure theory, it is convenient to introduce a notion of property that is true \emph{almost everywhere}.
If $\mu$ is a real measure over $\R$ and if $P$ is a property of real numbers, then one has the equivalences
$$
	\mu(\{x \in A : P(x) \text{ is true}\}) = 1 \Longleftrightarrow \mu(\{x \in A : P(x) \text{ is false}\}) = 0
$$
However, for the uniform measure on $\civita$ the above equivalence is in general false.
In \cite{shamseddine2012} it is explicitly observed that, given two functions $f, g : A \rightarrow \civita$, the two properties
\begin{enumerate}
	\item $\m(\{x \in A : f (x) = g(x)\}) = \m(A)$ and
	\item $\m(\{x \in A : f (x) \ne g(x)\}) = 0$
\end{enumerate}
are not equivalent: while (1) implies (2), in general (2) does not imply (1).
An example is obtained by choosing $A=[0,1]_\civita$, $f=\chi_{[0,1]_\civita}$ and $g=f-\chi_{[0,1]_{\Q}}$: these functions satisfy (2), but not (1).
This happens because $\{x \in A : f (x) = g(x)\} = [0,1]_\civita\setminus [0,1]_\Q$ and we have already recalled that the set $[0,1]_\civita \setminus [0,1]_{\Q}$ is not measurable.
In \cite{shamseddine2012} it is suggested that, in the case of equality, it is more convenient to use definition (1) instead of the weaker (2).
Inspired by this choice, we define the notion of property that is true almost everywhere on $A$ as a property that is true on a measurable $B \subseteq A$ with $\m(B)=\m(A)$.

\begin{definition}\label{def ae}
		Aa property $P$ holds almost everywhere on a measurable set $A\subset \civita$ if and only if $T_P=\{x\in A: P(x)\text{ is true}\}$ is measurable and $\m\left(T_P\right) = \m(A)$.
\end{definition}

Notice that, as a consequence of this definition, the two assertions ``$f=g$ a.e.\ on $A$'' and ``$f\ne g$ a.e.\ on $A$'' are not equivalent.

Thanks to the properties of the measure, if two properties hold almost everywhere on $A$, then both their conjunction and their disjunction hold almost everywhere on $A$.

\begin{lemma}\label{lemma prop a.e.}
	If $P$ and $Q$ hold almost everywhere on a measurable set $A\subset \civita$, then both $P \land Q$ and $P \lor Q$ hold almost everywhere on $A$.
\end{lemma}
\begin{proof}
	Define $T_P$ and $T_Q$ as in Definition \ref{def ae}.
	Then $T_P$ and $T_Q$ are measurable and $\m(T_P)=\m(T_Q)=\m(A)$.
	Moreover, $T_{P\land Q} = T_{P} \cap T_Q$, and $T_{P \lor Q} = T_P \cup T_Q$.
	Both $T_{P\land Q}$ and $T_{P \lor Q}$ are measurable as a consequence of Propositions 2.7 and 2.6 of \cite{berz+shamseddine2003}.
	Since $\m(T_P \cup T_Q) = \m(T_P) + \m(T_Q) - \m(T_P \cap T_Q)$ and since $\m(T_P) = \m(T_Q) = \m(A)$, we deduce that also $\m(T_P \cap T_Q) = \m(T_P \cup T_Q) = \m(A)$.
\end{proof}

We conclude our discussion of measurable sets by showing that,
if a set $A$ is measurable and if $\{I_n\}_{n \in \N}$ is a sequence of pairwise disjoint intervals such that the measure of $\bigcup_{n \in \N} I_n$ is of the same magnitude as the measure of $A$, then there must be at least an interval $I_n$ whose measure has the same magnitude as the measure of $A$.

\begin{lemma}\label{lemma dimensione insiemi}
	Let $A \subset \civita$ be measurable.
	If $\lambda(\m(A))=q$ and if $\{I_n\}_{n \in \N}$ is a sequence of pairwise disjoint intervals satisfying
	\begin{enumerate}
		\item $\bigcup_{n \in \N} I_n \subseteq A$;
		\item $\sum_{n \in \N} l(I_n)$ strongly converge;
		\item $\lambda\left(\m(A)- \sum_{n \in \N} l(I_n)\right)>q$;
	\end{enumerate}
	then
	\begin{enumerate}[label=(\roman*)]
		\item $\lambda\left(l(I_n)\right) \geq q$ for all $n \in \N$ and
		\item there exists $n \in \N$ such that $\lambda\left(l(I_n)\right)=q$.
	\end{enumerate}
\end{lemma}
\begin{proof}
	Both results can be obtained as a consequence of Lemma \ref{lemma tecnico serie a termini positivi}: in this case, $a_n = l(I_n)$ and $l = \sum_{n \in \N} l(I_n)$.
	Notice that hypothesis (3) implies that $\lambda\left(\sum_{n \in \N} l(I_n)\right) = \lambda(\m(A))=q$.
\end{proof}

\subsection{Measurable functions}\label{section measurable functions}

In analogy with the Lebesgue measure, the family of measurable functions is obtained from a family of simple functions.

For the measure on the Levi-Civita field, a meaningful choice is not the family of step functions, that are instead used for the the definition of Lebesgue measurable functions.
This choice and the properties of the order topology of $\civita$
would lead to a narrow class of measurable functions.

Instead, it has been proposed a different notion of simple functions: these can be chosen to be any family of continuous functions that have an antiderivative and with some monotonicity requirements.

\begin{definition}\label{simple functions}
	A family of functions defined over $I(a,b)$ is called simple iff
	\begin{enumerate}
		\item it is an algebra over $\civita$ that contains the identity function;
		\item every function of the family is continuous and has an antiderivative;
		\item every differentiable simple function with null derivative is constant, and differentiable simple functions with non-negative derivative are nonincreasing.
	\end{enumerate}
\end{definition}

It has been observed by Shamseddine and Berz that requirement (3), which is trivially satisfied by any real differentiable function, might not be satisfied for some functions defined on the Levi-Civita field. For some examples and a more detailed discussion, we refer to \cite{berz+shamseddine2003}.

From Definition \ref{simple functions} it can be readily obtained that the smallest family of simple functions is the algebra of polynomials over $I(a,b)$.
However, an even richer family of simple functions, already proposed by Shamseddine and Berz, is the family of Lipschitz continuous power series that weakly converge on some closed interval $[a,b]_{\civita}$.

For the remainder of the paper, we will work with the family of simple functions $\P=\bigcup_{a,b \in \civita, a<b}\P(I(a,b))$: as a consequence, a function $f$ is simple iff there exists an interval $I\subset \civita$ such that $\supp f= I$ and $f$ is a power series that converges for every $x\in I$.
We remark that we do not require Lipschitz continuity, assumed in \cite{shamseddine2012, berz+shamseddine2003} and used to prove that a simple function defined over an interval $(a,b)_\civita$ has a continuous extension to the interval $[a,b]_\civita$. However, thanks to Proposition \ref{lemma divergenza}, if a function is simple on an open interval $(a,b)_\civita$, then it has a simple extension to the closed interval $[a,b]_\civita$.

\begin{proposition}\label{prop simple}
	If $f$ is simple on $I(a,b)$, then there exists a unique simple function $g:[a,b]_\civita\rightarrow \civita$ such that $g_{|I(a,b)}=f$.
\end{proposition}
\begin{proof}
	Let $f(x) = \sum_{n \in \N} a_n \frac{(x-x_0)^n}{n!}$ for some $a_n \in \civita$.
	If $\sum_{n \in \N} a_n \frac{(a-x_0)^n}{n!}$ does not converge, then by Proposition \ref{lemma divergenza} there exists $h \in M_o$, $h>0$ such that $\sum_{n \in \N} a_n \frac{(a+h-x_0)^n}{n!}$ does not converge.
	However, this would contradict that $f$ is simple, i.e.\ analytic, on $I(a,b)$.
	A similar argument applies to $\sum_{n \in \N} a_n \frac{(b-x_0)^n}{n!}$.
	
	Since both $\sum_{n \in \N} a_n \frac{(a-x_0)^n}{n!}$ and $\sum_{n \in \N} a_n \frac{(b-x_0)^n}{n!}$ converge weakly, the desired extension of $f$ to $[a,b]_\civita$ is
	$$
		g(x) = \left\{
		\begin{array}{ll}
		f(x) & \text{if } x \in I(a,b)\\
		\sum_{n \in \N} a_n \frac{(a-x_0)^n}{n!} & \text{if } x = a\\
		\sum_{n \in \N} a_n \frac{(b-x_0)^n}{n!} & \text{if } x = b.
		\end{array}
		\right.
	$$
	As a consequence of this definition, $g$ is a power series, thus simple, and unique.
\end{proof}

With a slight abuse of notation, if $f$ is simple on $I(a,b)$, we will still denote by $f$ the simple function defined on $[a,b]_\civita$ that coincides with $f$ on $I(a,b)$.

From the algebra of simple functions it is possible to define the family of measurable functions.
In contrast to what happens in \cite{shamseddine2012, berz+shamseddine2003}, we will not require that measurable functions must be bounded.

\begin{definition}\label{measurable functions}
	Let $A\subset \civita$ be measurable and let $f: A \rightarrow \civita$.
	The function $f$ is measurable iff for all $\varepsilon\in \civita$, $\varepsilon > 0$, there exists a sequence of mutually disjoint intervals $\{I_n\}_{n\in\N}$ such that
	\begin{enumerate}
		\item $\bigcup_{n\in\N} I_n \subseteq A$;
		\item $\sum_{n\in\N} l(I_n)$ strongly converges in $\civita$;
		\item $\m(A)-\sum_{n\in\N} l(I_n)<\varepsilon$;
		\item for all $n \in \N$, $f$ is simple on $I_n$.
	\end{enumerate}

	We will denote by $\meas(A)$ the set of measurable functions on $A$.
\end{definition}

By removing the boundedness hypothesis from the definition of measurable functions, in principle some of the known results on measurable functions might not be valid in our development of the uniform measure over $\civita$.
However, this hypothesis is not used to prove any results of  \cite{berz+shamseddine2003} up to the definition of integral of a measurable function.

We recall here the main properties of measurable functions.

\begin{lemma}\label{3.13}
	For all measurable $A \subset \civita$, $\meas(A)$ is a vector space over $\civita$.
\end{lemma}
\begin{proof}
	See Proposition 3.9 of \cite{berz+shamseddine2003}.
	Notice that the proof of this proposition does not depend upon the hypothesis that simple functions are Lipschitz continuous or that measurable functions are bounded.
\end{proof}

In \cite{berz+shamseddine2003} it is given a necessary condition for a function $f$ to be measurable.

\begin{proposition}\label{measurable implies locally simple}
	A measurable function on $A$ is locally a simple function almost everywhere on $A$.
\end{proposition}

As a consequence of this property, measurable functions are continuous almost everywhere. Moreover, if two measurable functions are also differentiable and have the same derivative, then they differ by a constant. For the proofs of these statements, we refer to \cite{berz+shamseddine2003}.

Proposition \ref{measurable implies locally simple} can be further sharpened by using Lemma \ref{lemma dimensione insiemi}.

\begin{proposition}\label{measurable implies simple on a big set}
	If $A\subset\civita$ is a measurable set
	and if $f: A \rightarrow \civita$ is measurable, then $f$ is simple on an interval $I\subseteq A$ with $\lambda(l(I))=\lambda(m(A))$.
\end{proposition}
\begin{proof}
	Let $\{I_n\}_{n \in\N}$ be a sequence of pairwise sets satisfying Definition \ref{measurable functions} and hypotheses (1)-(3) of Lemma \ref{lemma dimensione insiemi}.
	Lemma \ref{lemma dimensione insiemi} entails the existence of some $n \in\N$ such that $\lambda(l(I_n))=\lambda(\m(A))$, as desired.
\end{proof}

A partial converse of Proposition \ref{measurable implies locally simple} is that every real function that is not locally analytic on an interval cannot be obtained as the restriction of a measurable function.

\begin{proposition}\label{lemma not analytic}
	If $f$ is not locally analytic at any point in $[a,b]_{\R}$, then for every $g\in \meas([a,b]_\civita)$, there exists an interval $I \subseteq [a,b]_{\R}$ such that $f(x)\not \approx g(x)$ for almost every $x \in I$.
\end{proposition}
\begin{proof}
	Let $g\in\meas([a,b]_\civita)$ and let $\{I_n\}_{n \in \N}$ satisfy conditions 1-4 of Definition \ref{measurable functions} for $g$ and for some $\varepsilon \approx 0$.
	As a consequence $\lambda(l(I_n))\geq 0$ for all $n \in \N$.
	By Lemma \ref{lemma dimensione insiemi}, there is $n \in \N$ such that $\lambda(l(I_n))=0$: otherwise the condition of strong convergence of $\sum_{n \in \N} l(I_n)$ and the choice of $\varepsilon \approx 0$ would entail that $\m\left(\bigcup_{n\in\N} I_n \right)\approx 0$, against the hypothesis $(b-a)-\m\left(\bigcup_{n\in\N} I_n \right)<\varepsilon$.
	
	Let $r \in I_n \cap \R$ and let $g(x)=\sum_{n \in \N} a_n (x-r)^n$ whenever $x \in I_n$.
	If $\lambda(a_n)<0$ for some $n \in \N$, then for every $x \in I_n$, $x \not \sim r$, $\lambda(g(x))<0$.
	Thus $g(x) \not \approx f(x)$ for every $x \in I_n \cap \R$, $x \ne r$.
	As a consequence, the desired interval $I$ can be obtained as one of the connected components of $(I_n \setminus \{r\}) \cap \R$.
	
	Suppose now that $\lambda(a_n) \geq 0$ for all $n\in\N$.
	We want to prove that the hypothesis $f(x)\approx g(x)$ for all $x \in I_n\cap \R$ entails that $f$ is analytic at every point in $I_n\cap \R$, against the hypotheses.

	If $a_n \in \R$ for all $n \in \N$, then the desired result is trivially true.
	
	If $\lambda(a_n)\geq 0$ for all $n \in \N$, then $a_n \approx a_n[0]$ for all $n \in \N$.
	Hence, if $x \in I_n \cap \R$, then $a_0[0] \in \R$ and $\sum_{n \in \N} a_n[0](x-r)^n\in\R$, so that $f(r) = a_0[0] \approx g(r)$ and $f(x)=\sum_{n \in \N} a_n[0](x-r)^n \approx g(x)$ for all $x \in I_n \cap \R$.
	As a consequence, $f$ is locally analytic at $r$, a contradiction.
\end{proof}

As observed in \cite{shamseddine2012,berz+shamseddine2003}, a function that is locally simple almost everywhere on $A$ is in general not measurable.
As a novel counterexample, we will prove that the order $n$ continuation of any function that is not locally analytic is not measurable on any measurable set of non-infinitesimal length, even if it is locally a simple function at every point of its domain.

\begin{lemma}\label{lemma order n extension not measurable}
	If $f \in C^n([a,b]_{\R})$ and if $f$ is not locally analytic at any point in $[a,b]_{\R}$, then $\ext{n}{f}\not\in\meas(A)$ for any measurable set $A\subseteq [a,b]_\civita$ with $\m(A)\not\approx0$.
\end{lemma}
\begin{proof}
	This Lemma can be obtained as a consequence of Proposition \ref{lemma not analytic}. We provide another proof that does not depend upon this result.
	
	By definition of the order $n$ extension of $f$ and by the hypothesis that $f$ is not locally analitic in $[a,b]_{\R}$, for every $I\subseteq [a,b]_\civita$ with $l(I)\not \approx 0$, $f$ is not simple over $I$.
	Consequently, if a family of mutually disjoint intervals $\{I_n\}_{n \in \N}$ satisfies condition 4 of Definition \ref{measurable functions}, then $l(I_n)\approx 0$ for all $n \in\N$.
	
	Suppose towards a contradiction that $\sum_{n\in\N} l(I_n)$ strongly converges to $\m(A)$.
	Lemma \ref{lemma dimensione insiemi} entails that there exists $m \in \N$ such that $\lambda(l(I_m)) = \lambda(A)$, contradicting $l(I_n)\approx 0$ for all $n \in \N$.
	Consequently, if $\sum_{n\in\N} l(I_n)$ strongly converges, it does not converge to $\m(A)$.
	We deduce that $\ext{n}{f}\not\in\meas(A)$, as desired.
\end{proof}

Since the functions $\ext{n}{f}$ are continuous but not measurable, we deduce that some continuous functions on the Levi-Civita field are not measurable.
This is in contrast with the well-known result that all real continuous functions are Lebesgue measurable.

Another class of locally simple functions that are not measurable are the reciprocal of simple functions on a neighbourhood of one of its zeroes.

\begin{proposition}\label{prop nonmeas}
	If $f: [a,b]_\civita \rightarrow \civita$ is a simple function, if $f(x) \ne 0$ for all $x \ne a$ and if $f(a)=0$, then the function $x \mapsto \frac{1}{f(x)}$ is not measurable on any interval of the form $I(a,c)\subseteq(a,b]_\civita$ with $a<c\leq b$.
\end{proposition}
\begin{proof}
	Let $g : (a,b]\rightarrow \civita$ be defined by $g(x)=\frac{1}{f(x)}$.
	Let also $\{I_n\}_{n\in\N}$ be a family of pairwise disjoint intervals such that $g$ is simple over $I_n$.
	We claim that for every $n \in \N$ $I_n \ne I(a,c)$ for any $a<c\leq b$: otherwise there would be a power series $\sum_{n \in \N} a_n (x-x_0)^n$ defied over an interval $I(a,c)$ such that $g(x) = \sum_{n \in \N} a_n (x-x_0)^n$ whenever $x-x_0 \in I(a,c)$.
	However, the hypotheses over $f$ entail that $\sum_{n \in \N} a_n (a-x_0)^n$ does not converge, so that by Proposition \ref{lemma divergenza} there exists $s \in \Q$ such that for every $h \in \civita$ with $h>0$ and $\lambda(h)>s$, $\sum_{n \in \N} a_n (a+h-x_0)^n$ does not converge, against the hypothesis that $\sum_{n \in \N} a_n (x-x_0)^n$ is convergent over $(a,c)_\civita$.
\end{proof}

We conclude with an example of a nonmeasurable, locally analytic function that does not belong to the categories of nonmeasurable functions discussed above.

\begin{example}\label{remark counterexample loc an not measurable}
	Another locally analytic function that is not measurable is constructed as follows. Let $A=\left[0,d^{-1}\right]$ and let $f(x)=\sin\left(x_{|[0,\infty]_\Q}\right)$. Since $f$ is locally a composition of analytic functions, $f$ is locally analytic \cite{analytic1}.
	However, $f$ is not measurable.
	In order to see that this is the case, consider the sets
	$$G(r,q)=\{y \in \civita: \lambda(rd^q-y) \geq 0 \} = \{y \in \civita: y=rd^q+h \text{ for some } h\in\civita \text{ with } \lambda(h)\geq 0 \},$$ defined for $r \in \R\setminus \{0\}$ and for $q\in(-1,0)_{\Q}$. They satisfy the following properties:
	\begin{itemize}
		\item $\bigcup_{r \in \R\setminus \{0\},\ q \in (-1,0)_{\Q}} G(r,q) \subset [0,d^{-1}]$;
		\item if $r,s \in \R\setminus \{0\}$ and $r \ne s$, then $G(r,p) \cap G(s,q) = \emptyset$ for all $p,q \in (-1,0)_{\Q}$;
		\item for every $r \in \R\setminus \{0\}$ and $q\in(-1,0)_{\Q}$, $f$ is analytic over $G(r,q)$ and it is not analytic over any set $S \supset G(r,q)$;
		\item by Lemma \ref{lemma monad}, each set $G(r,q)$ is not measurable, but for every $n\in\N$ it contains measurable sets of measure at least $n$.
	\end{itemize}
	As a consequence of these properties, any family of disjoint intervals satisfying conditions (1) and (4)
	of Definition \ref{measurable functions} must either include a refinement of the uncountable family $\{G(r,q)\}_{r\in\R\setminus \{0\},\ q \in(-1,0)_{\Q}}$ or not satisfy condition (3) of Definition \ref{measurable functions}. 
\end{example}

\subsection{Integrals}\label{sec integrals}

Since condition 2 of Definition \ref{simple functions} ensures that simple functions have an antiderivative, it is possible to define the integral of a simple function over an interval by imposing the validity of the fundamental theorem of calculus.
The integral of a measurable function over a measurable set can then be obtained as a limit of the integrals of simple functions over a sequence of intervals satisfying Definition \ref{measurable functions}.

\begin{definition}\label{def integral}
	If $f$ is a simple function over $I(a,b)$ whose antiderivative is $F$, then
	$$
		\int_{I(a,b)} f(x)  = \lim_{x\rightarrow b} F(x) - \lim_{x \rightarrow a} F(x).
	$$
	Notice that the two limits in the previous equality are well-defined, since $F$ is simple on $I(a,b)$ and, thanks to Proposition \ref{prop simple}, $F$ can be extended to a simple function on $[a,b]_\civita$.
	
	If $A \subset \civita$ is a measurable set and $f: A \rightarrow \civita$ is a measurable function, then define
	$$
		\mathcal{I}(f,A) = \left\{ \{I_n\}_{n\in\N} : \bigcup_{n \in \N} I_n \subseteq A,\ I_n \text{ are mutually disjoint and } f \text{ is simple on } I_n\ \forall n \in \N \right\}.
	$$
	The integral of $f$ over $A$ is defined as
	$$
		\int_{A} f(x)  = \lim_{\{I_n\}_{n\in\N}\in \mathcal{I}(f,A),\ \sum_{n \in \N} I_n \rightarrow \m(A)} \left( \sum_{n \in \N}  \int_{I_n} f(x) \right)
	$$
	whenever the limit on the right side of the equality is defined (and possibly equal to $\pm \infty$ whenever the sequence $k \mapsto \sum_{n \leq k}  \int_{I_n} f(x)$ diverges), and it is undefined otherwise.
\end{definition}

\begin{remark}
	Even if there exist some measurable functions with an infinite or undefined integral, all of the results on the integrals of measurable functions discussed in  \cite{berz+shamseddine2003,shamseddine2012} hold for every function with a well-defined integral.
	For instance, if $f$ and $g$ have a well-defined and finite integral over the disjoint measurable sets $A$ and $B$, then for every $a, b \in \civita$
	$$
		\int_A (af + bg) = a \int_A f + b \int_A g
		\text{ and }
		\int_{A \cup B} f = \int_A f + \int_B f.
	$$
\end{remark}

The only assertion that needs to be verified for unbounded measurable functions is Theorem 3.7 of \cite{shamseddine2012}.
The next proposition ensures that it is still true for measurable unbounded functions.

\begin{proposition}\label{proposizione uguali q.o.}
	Let $A \subset \civita$ be measurable and let $f, g : A \rightarrow \civita$ satisfy $f = g$
	a.e. on $A$. Then $f$ is measurable on $A$ if and only if $g$ is measurable on $A$; moreover
	\begin{enumerate}
		\item $\int_A f$ is defined if and only if $\int_A g$ is defined;
		\item if $\int_A f$ is defined, then $\int_A f = \int_A g$ (including the case $\int_A f = \int_A g = \pm \infty$).
	\end{enumerate}
\end{proposition}
\begin{proof}
	If $f$ and $g$ are bounded, then the desired result is a consequence of Theorem 3.7 of \cite{shamseddine2012}.
	
	We only need to address the case when $f$ and/or $g$ are not bounded.
	As in the proof of Theorem 3.7 of \cite{shamseddine2012}, let $B = \{ x \in A : f(x) = g(x)\}$ and $S=A\setminus B$.
	Then $B$ and $S$ are measurable, $\m(B)=\m(A)$ and $\m(S) = 0$.
	Since $f = g$ on $B$, it is clear that $\int_B f$ is defined if and only if $\int_B g$ is defined and that $\int_B f = \int_B g$.
	
	In light of Proposition 4.13 of \cite{berz+shamseddine2003}, it is sufficient to prove that for every measurable function $F: A \rightarrow \civita$ such that $\int_A F$ is defined, then $\int_S F = 0$ for every $S \subset A$ with $\m(S) = 0$.
	However, this is a consequence of the fact that, if $\m(S) = 0$, then the family $\mathcal{I}(f,S)$ introduced in Definition \ref{def integral} contains only the empty set.
\end{proof}

Even if measurable functions are not bounded, continuous measurable functions over intervals are bounded almost everywhere.

\begin{lemma}\label{lemma bounded ae}
	Let $f: I(a,b)\rightarrow \civita$ be measurable and continuous.
	Then $f$ is bounded almost everywhere in $I(a,b)$.
\end{lemma}
\begin{proof}
	Let $\varepsilon \in \civita$, $\varepsilon > 0$ and let $\{I(a_n,b_n)\}_{n\in\N}$ be a sequence of mutually disjoint intervals satisfying conditions (1)--(4) of Definition \ref{measurable functions}.
	By Lemma 4.7 of \cite{reexp}, $f$ is bounded on each interval $I(a_n, b_n)$.
	Moreover, by Proposition \ref{prop simple}, for every $n\in\N$ $\lim_{x\rightarrow a_n^+} f(x) \in \civita$ and $\lim_{x\rightarrow b_n^-} f(x) \in \civita$, i.e. both limits exist and are finite.
	By continuity of $f$, whenever $a_m = b_n$ for some $m, n \in \N$, then $\lim_{x\rightarrow a_m^-} f(x) = \lim_{x \rightarrow b_n^+} f(x) \in \civita$.
	As a consequence, $f$ is bounded on $\bigcup_{n\in\N} I(a_n, b_n)$.
	By the arbitrariness of $\varepsilon$, $f$ is bounded almost everywhere on $I(a,b)$.
\end{proof}

Thanks to the previous result, we can show that the antiderivative of a continuous, measurable function $f$ can be defined from the integral of $f$.

\begin{proposition}\label{proposition fundamental theorem of calculus}
	Let $f: I(a,b)\rightarrow \civita$ be measurable and continuous.
	Then the function $F: I(a,b) \rightarrow \civita$ defined by $F(x) = \int_{I(a,x)} f$ is well-defined and measurable.
	Moreover, $F'(x) = f(x)$ for almost every $x \in I(a,b)$.
\end{proposition}
\begin{proof}
	We begin our proof by showing that $F(x)$ is defined for every $x\in I(a,b)$.
	By Proposition 3.8 of \cite{berz+shamseddine2003}, $f$ is measurable on each set of the form $I(a,x)$ for $x\in I(a,b)$.
	Since $f$ is continuous, by Lemma \ref{lemma bounded ae} it is bounded almost everywhere in $I(a,b)$.
	Thus there exists a bounded measurable function $g$ over $I(a,b)$ such that $g(x)=f(x)$ for almost every $x \in I(a,b)$.
	Corollary 4.6 of \cite{berz+shamseddine2003} entails that $\int_{I(a,x)} g$ is defined for every $x \in I(a,b)$.
	By Proposition \ref{proposizione uguali q.o.}, we conclude that $F(x)=\int_{I(a,x)} f$ is also defined and not infinite for every $x\in I(a,b)$.
	
	We will now prove that $F$ is measurable on $I(a,b)$.
	Let $\varepsilon \in \civita$, $\varepsilon > 0$ and let $\{I_n\}_{n\in\N}$ be a sequence of mutually disjoint intervals satisfying conditions (1)--(4) of Definition \ref{measurable functions} for $f$.
	Then for every $n \in \N$ the restriction of $f$ to $I_n$ is analytic.
	As a consequence, for every $n\in\N$ there exists an antiderivative $F_n$ of $f$ such that for every $c, d \in I_n$, $c<d$, $\int_{[c,d]_\civita} f = F_n(d) - F_n(c)$.
	Since $\int_{[c,d]_\civita} f = \int_{[a,d]_\civita} f - \int_{[a,c]_\civita} f$, $F-F_n$ is constant over $I_n$, so that $F$ is a difference of simple functions. By Lemma \ref{3.13}, $F$ is also simple on each of the intervals $I_n$.
	We conclude that $F$ satisfies Definition \ref{measurable functions}, hence it is measurable.
	
	In order to prove that $F'(x) = f$ for almost every $x \in I(a,b)$, consider $\varepsilon$ and $\{I_n\}_{n\in\N}$ as in the previous part of the proof.
	Let also $x \in I_n$ such that there exist $h_x > 0$ that satisfies the inclusion $[x,x+h]\subset I_n$.
	As a consequence, if $0<h<h_x$,
	$$
		\frac{F(x+h)-F(x)}{h} = \frac{1}{h} \int_{[x,x+h]} f.
	$$
	Let $m_h = \min_{y\in[x,x+h]}\{f(y)\}$ and $M_h = \max_{y\in[x,x+h]}\{f(y)\}$.
	Notice that $m_h$ and $M_h$ exist by Theorem 3.3 of \cite{evt}.
	Moreover, continuity of $f$ entails that $\slim_{h\rightarrow0^+} m_h = \slim_{h\rightarrow0^+} M_h = f(x)$.
	However,
	$$
		m_h \leq \frac{1}{h} \int_{[x,x+h]} f \leq M_h,
	$$
	so that, by the squeeze theorem for strong convergence,
	$$
		\slim_{h\rightarrow0^+} \frac{F(x+h)-F(x)}{h}
		=
		\slim_{h\rightarrow0^+} \frac{1}{h} \int_{[x,x+h]} f
		=
		f(x).
	$$
	A similar argument applies to the case when there exist $h_x > 0$ that satisfies the inclusion $[x-h,x]\subset I_n$.
	Thus $F'(x) = f(x)$ for every $x \in \bigcup_{n\in\N} I_n$.
	By the arbitrariness of $\varepsilon$, we conclude that $F'(x) = f(x)$ almost everywhere in $I(a,b)$.
\end{proof}

We have recalled that Shamseddine and Berz proved that bounded measurable functions always have a well-defined integral. However, there are some unbounded measurable functions whose integral is infinite or undefined.

\begin{example}\label{example infty and nonmeas}
	Let $\{I_n\}_{n\in\N}$ be a sequence of mutually disjoint intervals satisfying
	\begin{itemize}
		\item $I_n \ne \emptyset$ for all $n \in \N$;
		\item $l(I_n)<d^{-n}$ for all $n \in\N$.
	\end{itemize}
	Let $A = \bigcup_{n \in \N}I_n$: by our choice of $\{I_n\}_{n\in\N}$, $A$ is measurable.
	
	Define $f: A \rightarrow \civita$ by $f(x) = \frac{1}{l(I_n)^2}$ whenever $x \in I_n$.
	Then $f$ is measurable by definition, but
	$$\int_A f(x) = \sum_{n \in \N} \int_{I_n} f(x) = \sum_{n \in \N} \frac{l(I_n)}{l(I_n)^2} = +\infty.$$
	
	Now define $g: A \rightarrow \civita$ by $g(x) = \frac{1}{l(I_n)}$ whenever $x \in I_n$.
	Then $g$ is also measurable, but
	$$\int_A g(x) = \sum_{n \in \N} \int_{I_n} g(x) = \sum_{n \in \N} \frac{l(I_n)}{l(I_n)}$$
	and the last series of the above equality is undefined, since it neither converges strongly nor diverges.
\end{example}

The existence of measurable functions with infinite or undefined integral is one of the motivations for the introduction of the $\L^p$ spaces over the Levi-Civita field.

\section{$\L^p$ spaces on the Levi-Civita field}\label{sec lp}

We have seen in the previous section that, according to Definitions \ref{measurable functions} and \ref{def integral}, some measurable functions do not have a well-defined integral or have an infinite integral.
This feature suggests that it would be meaningful to introduce spaces of functions whose $p$-th power is measurable and has a finite integral.
In order to do so, we need to be able to identify functions that are equal almost everywhere, as in the case of real Lebesgue measurable functions.
This identification can be successfully introduced as a consequence of Proposition \ref{proposizione uguali q.o.} and of the nontrivial fact that equality almost everywhere is an equivalence relation.

\begin{proposition}\label{prop = q.o. equivalenza}
	Let $A \subset \civita$ be a measurable set.
	The relation $\cong$ over $\meas(A)$ defined by $f \cong g$ iff $f=g$ a.e. is an equivalence relation.
\end{proposition}
\begin{proof}
	Clearly, $\cong$ is reflexive and symmetric.
	We only need to prove that it is transitive.
	Let $f_1 \cong f_2$ and $f_2 \cong f_3$ in $\meas(A)$.
	Define $A_1 = \{ x \in A : f_1(x) = f_2(x) \}$, $A_2 = \{ x \in A : f_2(x) = f_3(x) \}$ and $A_3 = \{ x \in A : f_3(x) = f_1(x) \}$
	Then $A_1 \cap A_2 \subseteq A_3 \subseteq A$. By Lemma \ref{lemma prop a.e.}, $A_1 \cap A_2$ is measurable and $\m(A_1 \cap A_2)=\m(A)$.
	Lemma \ref{lemma A c B c C} ensures then that $A_3$ is measurable and $\m(A_3) = \m(A)$, so that $f_1 = f_3$ a.e. on $A$.
\end{proof}

From now on, we will identify a measurable function $f:A\rightarrow\civita$ with the equivalence class
$$
[f] = \{ g : A\rightarrow\civita : f(x)=g(x) \text{ for almost every } x\in A\}
$$
and, with an abuse of notation, we will often write $f$ instead of $[f]$.

As a consequence of the above identification, it is possible to refine the notion of measurable function with the introduction of $L^p$ spaces.

\begin{definition}\label{def lp}
	Let $A \subset \civita$ be a measurable set.
	
	If $1\leq p < \infty$, we define
	$$\L^p(A)=\left\{ [f] : f\in \meas(A), f^p\in \meas(A), \int_A |f|^p \text{ is defined and } \int_A |f|^p  < + \infty \right\}.$$
	If $f \in \L^p(A)$, then we define $\norm{f}_p = \left(\int_A |f|^p \right)^{1/p}$.
	
	For $p=\infty$, we define a function $\norm{\cdot}_\infty : \meas(A) \rightarrow \civita \cup \{\text{undefined}\}$ by posing
	$$\norm{f}_\infty = \inf_{c\in\civita}\{|f(x)|\leq c \text{ for almost every } x\in A\}$$
	and we define $\L^\infty$ as the set of measurable functions with a well-defined $\infty$-norm:
	$$\L^\infty(A)=\left\{ [f] : f\in \meas(A) \text{ and } \norm{f}_\infty \ne \text{ undefined} \right\}.$$
\end{definition}

Notice that the functions $\norm{\cdot}_p$ are well-defined thanks to Proposition \ref{proposizione uguali q.o.} and to Proposition \ref{prop = q.o. equivalenza}.

\begin{remark}
	It is possible to extend Definition \ref{def lp} by considering the vector space over $\mathcal{C}$ of $\mathcal{C}$-valued measurable functions, where $\mathcal{C}$ denotes the complexification of $\civita$.
	For the properties of $\mathcal{C}$, we refer to \cite{berz,calculusnumerics,analysislcf,convergence} and references therein.
\end{remark}

\begin{remark}\label{remark linfty}
If a function belongs to $\L^\infty$, then it is essentially bounded, but the converse is in general false.
In fact, it is not possible to define the $\L^\infty$ space and the function $\norm{\cdot}_\infty$ by defining
$$\L^\infty(A)=\left\{ [f] : f \in \meas(A) \text{ and } f \text{ is essentially bounded} \right\}$$
and then posing
$$\norm{f}_\infty = \inf_{c\in\civita}\{|f(x)|\leq c \text{ for almost every } x\in A\}$$
for all $f \in \L^\infty(A)$.
Indeed, with such a definition $\norm{f}_\infty$ would not be defined for all functions in $\L^\infty$.

A counterexample is constructed as follows: let $I_n = \left[n-{d^n}, n+{d^n}\right]$.
Then $A=\bigcup_{n \in \N} I_n$ is measurable, and the function $f: A \rightarrow \civita$ defined by $f_{|I_n} = n$ is measurable and essentially bounded, since $|f(x)|<d^{-1}$ for all $x \in A$.
However, $\norm{f}_\infty$ is not well-defined, since 
$$\{c \in \civita: f(x) < c \text{ for almost every } x \in A\} = \{c \in \civita: \lambda(c)<0\},$$
and it is well-known that the latter set is bounded from below but has no infimum.
\end{remark}

\begin{remark}\label{lp ne measurable}
	The distinction between $\L^p$ functions and measurable ones is significant, i.e. $\L^p(A) \ne \meas(A)$.
	For instance, let $\{I_n\}_{n\in\N}$, $A$ and $g$ be defined as in Example \ref{example infty and nonmeas}.
	We have already argued that $g$ is measurable, but $g \not \in \L^p(A)$ for any $p \geq 1$, since $g$ is not bounded and
	$$\int_A |g(x)|^p = \sum_{n \in \N} \int_{I_n} |g(x)|^p = \sum_{n \in \N} \frac{l(I_n)}{l(I_n)^p},$$
	and the latter sum is undefined if $p=1$ and diverges $+\infty$ if $p > 1$.
	Similarly, if $a > 1$, $g^{\frac{1}{a}} \in \L^p(A)$ for all $1\leq p < a$.
\end{remark}

Many properties of the $\L^p$ spaces are a consequence of the following H\"older's inequality, that is similar to the one valid for real $L^p$ spaces.

\begin{lemma}[H\"older's inequality]\label{holder inequality}
	Suppose that $A \subset \civita$ is measurable.
	For every $1\leq p \leq q \leq \infty$ with $1/p+1/q=1$ (allowing also for $p=1$ and $q=\infty$), for every $f \in \L^p(A)$ and $g\in L^q(A)$, then $fg\in L^1(A)$ and
	$$
		\norm{fg}_1 \leq \norm{f}_p \norm{g}_q.
	$$
\end{lemma}
\begin{proof}
	If $fg(x)=0$ for almost every $x \in A$, then the proof is trivial.
	Suppose then that $fg(x) \ne 0$ on a measurable set of positive measure.
	
	If $p=1$ and $q=\infty$, then the result is a consequence of Corollary 4.12 of \cite{berz+shamseddine2003}.
	
	Otherwise, notice that the Young's inequality $ab\leq \frac{a^p}{p} + \frac{b^q}{q}$ is true for every $a, b \in \civita$.
	As a consequence, the usual proof for the H\"older inequality still applies to this setting: since $\norm{f}_p \ne 0$ and $\norm{g}_q \ne 0$ by hypothesis, then for every $x \in A$
	$$
		\frac{|f(x)g(x)|}{\norm{f}_p \norm{g}_q} \leq \frac{|f(x)|^p}{p\norm{f}_p^p} + \frac{|g(x)|^q}{q\norm{g}_q^q}.
	$$
	Corollary 4.11 of \cite{berz+shamseddine2003} entails the inequality
	$$
		\frac{1}{\norm{f}_p\norm{g}_q}\cdot \int_A |f(x)g(x)|  \leq  \frac{1}{\norm{f}_p^p} \cdot \int_A \frac{|f(x)|^p}{p}  + \frac{1}{\norm{g}_q^q} \cdot \int_A \frac{|g(x)|^q}{q} = 1,
	$$
	that is equivalent to the H\"older's inequality.
\end{proof}

\begin{corollary}\label{corollary norm}
	For all measurable sets $A\subset \civita$ and for all $1 \leq p \leq \infty$, the maps $\norm{\cdot}_p : \L^p(A) \rightarrow \civita$ are norms.
\end{corollary}
\begin{proof}
	We need to prove that, for all $f, g \in \L^p$ and for all $a \in \civita$,
	\begin{enumerate}
		\item $\norm{f+g}_p \leq \norm{f}_p + \norm{g}_p$;
		\item $\norm{af}_p = |a|\norm{f}_p$;
		\item $\norm{f}_p = 0$ iff $f=0$.
	\end{enumerate}
	
	Thanks to linearity of the integral on the Levi-Civita field, established in Proposition 4.3 of \cite{berz+shamseddine2003}, and thanks to the H\"older's inequality established in Lemma \ref{holder inequality}, the proof of property 1 is analogous to the proof that the norms of the usual $L^p$ spaces are subadditive.
	
	Property 2 can be obtained from the definition of integral and by Proposition 4.3 of \cite{berz+shamseddine2003}.
	
	Property 3 is a consequence of the identification of functions a.e.\ equal and of Theorem 3.7 of \cite{shamseddine2012}.
\end{proof}

A result by Shamseddine and Berz entails an inclusion between the $\L^p$ spaces.

\begin{lemma}\label{inclusion Linfty Lp}
	For all measurable sets $A\subset \civita$, if $f$ is bounded then $f\in \L^p(A)$ for all $1 \leq p$.
	As a consequence, $\L^\infty(A)\subseteq \L^p(A)$ for all $1 \leq p < \infty$.
\end{lemma}
\begin{proof}
	Recall that if $A$ is measurable, then $\m(A)\in\civita$ (i.e.\ the measure of $A$ is not infinite).
	It is easy to check that if $f$ is measurable, then $|f|$ and $|f|^p$ are also measurable for all $1 \leq p < \infty$.
	
	By Corollary 4.12 of \cite{berz+shamseddine2003}, for all bounded $f$, $$\int_A |f|^p  \leq M^p \cdot \m(A),$$
	where $M$ is a bound for $|f|$, so that $f \in \L^p(A)$ for all $1 \leq p < \infty$.
	
	If $f \in \L^\infty(A)$, we can choose $M = \norm{f}_\infty$.
\end{proof}

\begin{remark}\label{remark if bounded then trivial}
	Lemma \ref{inclusion Linfty Lp} entails that,
	under the hypothesis that measurable functions are bounded, then the theory of $\L^p$ spaces over the Levi-Civita field becomes trivial.
	In fact, all bounded measurable functions have a well-defined integral. Moreover, since the $p-$th power of a bounded measurable function is still bounded and measurable, we would have $\meas(A) = L^p(A)$ for all $1\leq p < \infty$.
\end{remark}

It is possible to prove that the usual inclusions between the real $L^p$ spaces are valid also for the $\L^p$ spaces.

\begin{lemma}\label{inclusion lp spaces}
	If $A \subset \civita$ is measurable and if $1\leq p \leq q \leq \infty$, then $\L^q(A)\subseteq \L^p(A)$.
\end{lemma}
\begin{proof}
	If $q = \infty$, then $f \in \L^p(A)$ by Lemma \ref{inclusion Linfty Lp}.
	Suppose then that $q < \infty$.
	
	Thanks to the H\"older's inequality, the proof is the same as in the case of the Lebesgue measure: if $f \in \L^q(A)$, then
	\begin{eqnarray*}
	\int_A |f|^p  &\leq & \left( \int_A |f|^{pq/p}   \right)^{p/q} \left(\int_A 1  \right)^{1-p/q}\\
	&=& \norm{f}_q^p \cdot \m(A)^{1-p/q}.
	\end{eqnarray*}
	The hypothesis $f \in L^q(A)$ and the fact that $\m(A)\in\civita$ for any measurable set $A$ entail that $\m(A)^{1-p/q} \in \civita$, so that indeed $f \in \L^p(A)$.
\end{proof}

We conclude with some coherence results for the integral on the Levi-Civita field and the Lebesgue integral.

\begin{lemma}\label{lemma coherence}
	If $a, b \in \R$ and $f \in \an([a,b])$, then $\ext{\infty}{f} \in \L^1([a,b]_{\civita})$ and
	$$\int_{[a,b]} \ext{\infty}{f}  = \int_{[a,b]} f(x) dx .$$
	Moreover, for every $c, d \in \civita$ satisfying $a \leq c < d \leq b$,
	$$\int_{I(c,d)}\ext{\infty}{f}  \approx \int_{[c[0],d[0]]} f(x) dx .$$
\end{lemma}
\begin{proof}
	If $f \in \an([a,b])$, then $\ext{\infty}{f}$ is simple over $[a,b]_\civita$.
	Let $F$ be an antiderivative of $f$: then, by Theorem 3.14 of \cite{analytic1}, $\ext{\infty}{F}$ is an antiderivative of $\ext{\infty}{f}$.
	By definition \ref{def integral},
	$$
	\int_{[a,b]} \ext{\infty}{f} 
	=
	\ext{\infty}{F}(b)-\ext{\infty}{F}(a),
	$$
	while from real analysis it is well-known that $$\int_{[a,b]} f(x) dx =F(b)-F(a).$$
	Since $a, b \in \R$, the desired equality is a consequence of the property $\ext{\infty}{F}(r)=F(r)$ for all $r \in \R$.
	
	As for the second result, by definition of $\ext{\infty}{F}$ we have $\ext{\infty}{F}(c)\approx F(c[0])$ and $F(d)\approx\ext{\infty}{F}(d[0])$: these hypotheses are sufficient to entail the desired result.
\end{proof}

From the previous lemma we deduce that the extension of a real analytic function in $L^p$ belongs to the space $\L^p$ also in the Levi-Civita field.

\begin{corollary}\label{corollario piccolo}
	If $a, b \in \R$ and $f \in \an([a,b])$, then $\ext{\infty}{f} \in \L^p(A)$ for all measurable $A \subseteq [a,b]_\civita$ and for all $1\leq p \leq \infty$.
	Moreover, $\norm{f}_p = \norm{\ext{\infty}{f}}_p$, where the former is the norm in the space $L^p([a,b])$ and the latter is the norm in $\L^p([a,b]_\civita)$.
\end{corollary}

In the previous results, it is essential that the intervals under consideration are closed.
In fact, as a consequence of Proposition \ref{prop nonmeas}, functions of the form $x^{-\frac{1}{a}}$ with $a > 1$ are not in $L^{p}(0,1)$ for any $p\geq 1$, even if their restriction to $\R$ is in $\an(0,1)\cap L^p(0,1)$ for $1\leq p < a$.

\subsection{Convergence in the $\L^p$ spaces}

In the $\L^p$ spaces it is possible to consider sequences of functions that converge with respect to the $\L^p$ norm.
Strong and weak convergence on the Levi-Civita field yield two corresponding notions of convergence in norm.

\begin{definition}\label{cauchy+convergence}
	Let $A \subset \civita$ be a measurable set and let $\{f_n\}_{n \in \N}$ be a sequence of functions with $f_n \in \L^p(A)$ for all $n \in \N$.
	
	The sequence $\{f_n\}_{n \in \N}$ is strongly Cauchy in $\L^p(A)$ iff
	$$\forall \varepsilon \in \civita, \varepsilon > 0, \exists n_\varepsilon \in \N: \forall m, n > n_\varepsilon\ \norm{f_n-f_m}_p < \varepsilon.$$
	If there exists $f \in \L^p(A)$ such that
	$\lim_{n\rightarrow\infty} \norm{f_n-f}_p = 0$,
	then we will say that $\{f_n\}_{n \in \N}$ converges strongly to $f$ with respect to the $\L^p$ norm.
	
	The sequence $\{f_n\}_{n \in \N}$ is weakly Cauchy in $\L^p(A)$ iff
	$$\forall \varepsilon \in \R, \varepsilon > 0, \exists n_{\varepsilon} \in \N : \forall m, n > n_{\varepsilon}\ \max_{q \in \Q, q\leq \varepsilon^{-1}} \norm{f_n-f_m}_p[q] < \varepsilon.$$
	If there exists $f \in \L^p(A)$ such that
	$\wlim_{n\rightarrow\infty} \norm{f_n-f}_p = 0$,
	then we will say that $\{f_n\}_{n \in \N}$ converges weakly to $f$ with respect to the $\L^p$ norm.
\end{definition}

As a consequence of the fact that every sequence that converges strongly also converges weakly, if a sequence is strongly Cauchy then it is also weakly Cauchy.

In the next sections we will show that the $\L^p$ spaces are not sequentially complete with respect to both notions of convergence.

\subsection{Sequential completeness of the $\L^p$ spaces with respect to strong convergence}\label{sec completeness strong}

Some results addressing sequential completeness of the space $\L^\infty$ with respect to strong convergence have been obtained in \cite{shamseddine2012,berz+shamseddine2003}.

\begin{lemma}\label{lemma linfty not complete}
	If $A \subset \civita$ is measurable and $\m(A) > 0$, then $\L^\infty(A)$ is not sequentially complete.
\end{lemma}
\begin{proof}
	For a proof that $\L^\infty([0,1])$ is not sequentially complete, see example 3.10 of \cite{berz+shamseddine2003}.
	From this example, one can obtain a proof that $\L^\infty(A)$ is not sequentially complete for every set $A \subset \civita$ by rescaling the argument of example 3.10 to an interval $I(a,b) \subseteq A$ with $a \ne b$.
\end{proof}

From this result, one immediately obtains that the $\L^p$ spaces are not sequentially complete with respect to strong convergence.

\begin{proposition}\label{main result strong convergence}
	For every measurable $A\subset\civita$ with $\m(A)>0$ and for every $1 \leq p \leq \infty$, the spaces $L^p(A)$ are not sequentially complete.
\end{proposition}
\begin{proof}
	By Lemma \ref{lemma linfty not complete}, $\L^\infty(A)$ is not sequentially complete.
	Let $\{f_n\}_{n\in\N}$ be a strongly Cauchy sequence in $\L^\infty(A)$ that does not converge in $\L^\infty(A)$, and let $f$ be its uniform limit. Recall that $f$ is not measurable.
	
	Lemma \ref{inclusion lp spaces} ensures that $\L^\infty(A) \subseteq \L^p(A)$; thus by the H\"older's inequality
	$$\norm{f_n-f_m}_p^p = \int |f_n-f_m|^p \leq \norm{|f_n-f_m|^p}_\infty \cdot \m(A) = \norm{f_n-f_m}_\infty^p \cdot \m(A).$$
	Since $\{f_n\}_{n\in\N}$ is strongly Cauchy in $\L^\infty(A)$, from the previous inequality and from the squeeze theorem for the strong convergence we obtain that $\{f_n\}_{n\in\N}$ is strongly Cauchy also in $\L^p(A)$ for every $1 \leq p \leq \infty$.
	However, since $f \not \in \meas(A)$, $\{f_n\}_{n\in\N}$ is strongly Cauchy but does not converge in $\L^p(A)$ for any $1 \leq p \leq \infty$.
	We conclude that the spaces $\L^p(A)$ are not sequentially complete.
\end{proof}

Despite these negative results, strong convergence of a sequence $\{f_n\}_{n\in\N}$ to a function $f$ in $\L^\infty$ implies strong convergence of $\{f_n\}_{n\in\N}$ to $f$ in $\L^1$.

\begin{theorem}
	If $A \subset \civita$ is measurable and if $\{f_n\}_{n \in \N}$
	strongly converges to $f$ in $\L^\infty(A)$, then 
	$f\in \L^1(A)$ and $\lim_{n \rightarrow \infty} \int_A f_n  = \int_A f$.
\end{theorem}
\begin{proof}
	See Theorem 3.9 of \cite{shamseddine2012}.
\end{proof}

This result can be extended to show that strong convergence of a sequence $\{f_n\}_{n\in\N}$ to a function $f$ in $\L^\infty$ implies strong convergence of $\{f_n\}_{n\in\N}$ to $f$ in $\L^p$.

\begin{proposition}\label{main result strong convergence 2}
	If $A \subset \civita$ is measurable and if $\{f_n\}_{n \in \N}$ 
	strongly converges to $f$ in $\L^\infty(A)$, then $f\in \L^p(A)$ and $\lim_{n \rightarrow \infty} \norm{f_n - f}_p  = 0$ for every $1 \leq p \leq \infty$.
\end{proposition}
\begin{proof}
	We proceed as in the proof of Proposition \ref{main result strong convergence}.
	The fact that $f \in \L^p(A)$ is a consequence of Lemma \ref{inclusion lp spaces}.
	By the H\"older's inequality,
	$$\norm{f_n-f_m}_p  \leq \norm{f_n-f}_\infty^p \cdot \m(A).$$
	Since $\slim_{n \rightarrow \infty} \norm{f_n-f}_\infty = 0$, we have also $\slim_{n \rightarrow \infty} \norm{f_n-f}_\infty^p = 0$.
	By the squeeze theorem for the strong convergence in $\civita$, we obtain the desired conclusion.
\end{proof}

\subsection{Sequential completeness with respect to weak convergence}\label{sec completeness weak}

There are many arguments against sequential completeness of the $\L^p$ spaces with respect to weak convergence.
At a fundamental level, the field $\civita$ itself is not sequentially complete with respect to weak convergence: thus if $\{a_n\}_{n\in\N}$ is a weakly Cauchy sequence that does not converge in $\civita$ and if $A$ is a measurable subset of $\civita$, then the sequence of functions $\{f_n\}_{n\in\N}$ defined by $f_n(x) = a_n$ for all $x \in A$ satisfies
\begin{itemize}
	\item $f_n \in \L^p(A)$ for every $1\leq p < \infty$;
	\item $\{f_n\}_{n\in\N}$ is weakly Cauchy;
	\item $\{f_n\}_{n\in\N}$ does not converge in $\L^p(A)$, otherwise $\wlim_{n \rightarrow \infty} a_n \in \civita$, against the hypothesis.
\end{itemize}

Moreover, we have seen in Section \ref{sec completeness strong} that the $\L^p$ spaces are not even sequentially complete with respect to strong convergence.
Since strongly Cauchy sequences are also weakly Cauchy, then from Proposition \ref{main result strong convergence} one immediately obtains that the $\L^p$ spaces are not sequentially complete also with respect to weak convergence.

\subsection{A completion of the $\L^p$ spaces with respect to strong convergence}\label{section completion}

In Proposition \ref{main result strong convergence} we proved that the $\L^p$ spaces are not sequentially complete with respect to the strong or the weak convergence.
In this section we will introduce the completions of the $\L^p$ spaces with respect to strong convergence.
The resulting spaces will be sequentially complete with respect to strong convergence, and the completion of $\L^2$ will also have a well-defined inner product, making it a Hilbert space.

The completion of the $\L^p$ spaces with respect to strong convergence is obtained by considering all sequences of $\L^p$ functions that are strongly Cauchy with respect to the $p$-norm, and by identifying those sequences whose difference strongly converges to zero.

\begin{definition}\label{def pre slp}
	Let $A$ be a measurable set.
	Define
	$$B_s^p(A)=\{ \{f_n\}_{n\in\N} : f_n \in \L^p(A) \text{ for all } n \in\N \text{ and } \{f_n\}_{n\in\N} \text{ is strongly Cauchy}\},$$
	and define the relation $\cong_s$ over $B_s^p \times B_s^p$ by posing
	$$
		\{f_n\}_{n\in\N} \cong_s \{g_n\}_{n\in\N} \text{ iff } \slim_{n \rightarrow \infty} \norm{f_n-g_n}_p = 0.
	$$
\end{definition}

In order to consider the quotient with respect to the relation $\cong_s$, it is necessary to ensure that $\cong_s$ is an equivalence relation.

\begin{lemma}\label{lemma strong congruence}
	For every measurable $A \subset \civita$, the relation $\cong_s$ is an equivalence relation over $B_s^p(A)$.
\end{lemma}
\begin{proof}
	From the definition it is immediate to prove that $\cong_s$ is symmetric and reflexive.
	In order to prove that it is transitive, suppose that $\{f_n\}_{n\in\N} \cong_s \{g_n\}_{n\in\N}$ and that $\{g_n\}_{n\in\N} \cong_s \{h_n\}_{n\in\N}$.
	Since $\norm{f_n-z_n}_p = \norm{f_n-g_n+g_n-z_n}_p \leq \norm{f_n-g_n}+\norm{g_n-z_n}_p$, by the squeeze theorem for strong convergence we conclude that $\cong_s$ is also transitive.	
\end{proof}

Thanks to the above Lemma, we can define the completion of the $\L^p$ spaces with respect to strong convergence.

\begin{definition}\label{def slp}
	We define $\sL^p(A) = B_s^p(A) / \cong_s$.
	If $\{f_n\}_{n\in\N} \in B_s^p(A)$, we will denote its equivalence class in $\sL^p(A)$ by $\langle f_n \rangle$.
\end{definition}

These completions are vector spaces over $\civita$.

\begin{lemma}\label{lemma slp vector spaces}
	For every measurable $A \subset \civita$, $\sL^p(A)$ is a vector space over $\civita$.
\end{lemma}
\begin{proof}
	Corollary \ref{corollary norm} entails that $\sL^p(A)$ is closed under linear combinations: if $\{f_n\}_{n\in\N}$ and $\{g_n\}_{n\in\N}$ are sequences in $B_s^p(A)$ and if $\alpha, \beta \in \civita$, then the linear combination $\{\alpha f_n + \beta g_n\}_{n\in\N}$ satisfies the inequality
	$$
		\norm{\alpha f_n + \beta g_n}_p \leq |\alpha|\norm{f_n}_p + |\beta|\norm{g_n}_p.
	$$
	As a consequence,
	$$
	\norm{\alpha f_n + \beta g_n -(\alpha f_m + \beta g_m)}_p \leq |\alpha|\norm{f_n-f_m}_p + |\beta|\norm{g_n-g_m}_p.
	$$
	Since $\{f_n\}_{n\in\N}$ and $\{g_n\}_{n\in\N} \in B_s^p(A)$, we deduce that also $\{\alpha f_n + \beta g_n\}_{n\in\N} \in B_s^p(A)$, as desired.
\end{proof}

It is possible to introduce a norm on each of the $\sL^p$ spaces: the $\sL^p$ norm of $\langle f_n \rangle$ is defined as the strong limit of the $\L^p$ norms of $f_n$.
At first we will prove that this function does not depend on the representative of $\langle f_n \rangle$, and then that it is indeed a norm over $\sL^p$.

\begin{lemma}\label{lemma norma ben def}
	The function $\norm{ \cdot }_p : \sL^p(A) \rightarrow \civita$ defined by $\norm{\langle f_n \rangle}_p = \slim_{n \rightarrow \infty} \norm{f_n}_p$
	is well-defined, i.e.\ it does not depend on the representative of the equivalence class $\langle f_n \rangle$.
\end{lemma}
\begin{proof}
	Suppose that $\langle f_n \rangle = \langle g_n \rangle$ in $\sL^p(A)$.
	Then
	$\slim_{n \rightarrow \infty}\norm{f_n - g_n}_p = 0$, so that
	$$
	\slim_{n \rightarrow \infty}\norm{f_n}_p 
	=
	\slim_{n \rightarrow \infty} \norm{f_n - g_n + g_n}_p
	\leq
	\slim_{n \rightarrow \infty} \norm{f_n - g_n}_p + \slim_{n \rightarrow \infty} \norm{g_n}_p
	=
	\slim_{n \rightarrow \infty} \norm{g_n}_p.$$
	By reversing the roles of $\{f_n\}_{n\in\N}$ and $\{g_n\}_{n\in\N}$, it is possible to obtain the other inequality $\slim_{n \rightarrow \infty}\norm{g_n}_p \leq \slim_{n \rightarrow \infty}\norm{f_n}_p$.
	These inequalities entail also $\slim_{n \rightarrow \infty}\norm{f_n}_p = \slim_{n \rightarrow \infty}\norm{g_n}_p$.
	As a consequence, the function $\norm{ \cdot }_p$ does not depend upon the representative of $\langle f_n \rangle$.
\end{proof}

\begin{lemma}\label{lemma norm p}
	For every measurable $A \subset \civita$ and for every $1\leq p \leq \infty$, the function $\norm{ \cdot }_p : \sL^p(A) \rightarrow \civita$ is a norm over $\sL^p(A)$.
\end{lemma}
\begin{proof}
	We need to prove that, for all $\langle f_n \rangle, \langle g_n \rangle \in \L^p$ and for all $a \in \civita$,
	\begin{enumerate}
		\item $\norm{\langle f_n \rangle+\langle g_n \rangle}_p \leq \norm{\langle f_n \rangle}_p + \norm{\langle g_n \rangle}_p$;
		\item $\norm{a\langle f_n \rangle}_p = |a|\norm{\langle f_n \rangle}_p$;
		\item $\norm{\langle f_n \rangle}_p = 0$ iff $\langle f_n \rangle=0$.
	\end{enumerate}
	The first property
	is a consequence Corollary \ref{corollary norm} and of the fact that inequalities are preserved under strong limits.
	
	For the proof of the second property, notice that
	$$\norm{\langle af_n \rangle}_p =
	\slim_{n \rightarrow \infty} \norm{a f_n}_p =
	\slim_{n \rightarrow \infty} |a|\norm{f_n}_p =
	|a| \slim_{n \rightarrow \infty} \norm{f_n}_p = |a| \norm{\langle f_n\rangle}_p,$$
	so that $\norm{\langle af_n \rangle}_p = |a| \norm{\langle f_n\rangle}_p$.
	
	Finally recall that, by definition, if $\norm{\langle f_n \rangle} = 0$, then $\slim_{n\rightarrow \infty} \norm{f_n}_p = 0$, so  $\{f_n\}_{n\in\N} \cong_s \{0\}_{n\in\N}$, that is $\langle f_n \rangle = 0$.
\end{proof}

It turns out that the spaces $\sL^p$ are sequentially complete with respect to the norm defined above.

\begin{proposition}\label{proposition banach spaces}
	For every measurable $A \subset \civita$, $\sL^p(A)$ is sequentially complete with respect to the norm $\norm{\cdot}_p$.
	As a consequence, it is a Banach space over $\civita$.
\end{proposition}
\begin{proof}
	Taking into account Lemmas \ref{lemma slp vector spaces} and \ref{lemma norm p}, we only need to prove that $\sL^p(A)$ is sequentially complete with respect to the $\sL^p$ norm.
	The proof is analogous to the real case, and hinges upon sequential completeness of $\civita$ with respect to strong convergence.
	Let $\{\phi_k\}_{k\in\N}$ be a sequence in $\sL^p(A)$ such that
	\begin{equation}\label{eqn limite 0}
		\forall \varepsilon \in \civita, \varepsilon > 0\
		\exists k_\varepsilon \in \N : \forall h,k > k_\varepsilon\
		\norm{ \phi_h - \phi_k }_p < \varepsilon.
	\end{equation}
	
	By definition of $\sL^p(A)$, for every $k \in \N$ there exists $g_k \in \L^p(A)$ such that
	\begin{equation}\label{eqn limite 2}
		\norm{g_k - \phi_k}_p \leq d^{-k}.
	\end{equation}
	We claim that $\{g_n\}_{n\in\N}$ is strongly Cauchy in $L^p(A)$ and that
	$\slim_{k \rightarrow \infty}\norm{\phi_k - \langle g_n \rangle}_p = 0$.
	
	Let us show that $\{g_n\}_{n\in\N}$ is strongly Cauchy.
	Notice that, for every $h, k \in \N$,
	$$
		\norm{g_k - g_h}_p \leq
		\norm{g_k - \phi_k}_p
		+
		\norm{\phi_k - \phi_h}_p
		+
		\norm{\phi_h - g_h}_p
	$$
	Thanks to equation \eqref{eqn limite 2}, we have also $\norm{g_k - \phi_k}_p < d^{-k}$ and $\norm{\phi_h - g_h}_p < d^{-h}$.
	Let $\varepsilon \in \civita, \varepsilon > 0$.
	Then if $\max\{h,k\}>-\lambda(\varepsilon)$, $\norm{g_k - \phi_k}_p + \norm{\phi_h - g_h}_p < d^{-k} + d^{-h}<\varepsilon$.
	Similarly, if $h, k > k_\varepsilon$, by equation \eqref{eqn limite 0} we have also $\norm{f_{n,k} - f_{n,h}}_p < \varepsilon$.
	Thus
	$
		\norm{g_k - g_h}_p \leq 2\varepsilon,
	$
	and $\{g_k\}_{k\in\N}$ is strongly Cauchy in $\L_s^p(A)$.
	
	We now want to prove that $\{\phi_k\}_{k\in\N}$ converges to $\langle g_n \rangle$ in the $\sL^p$ norm.
	Observe that for every $k \in \N$
	$$\norm{\phi_k - \langle g_n \rangle}_p
	\leq
	\norm{\phi_k - g_k}_p + \norm{g_k - \langle g_n \rangle}_p
	$$
	By equation \eqref{eqn limite 2}, $\norm{\phi_k - g_k}_p < d^{-k}$ and, by definition of $\{g_k\}_{k\in\N}$, $\slim_{k \rightarrow \infty} \norm{g_k - \langle g_n \rangle}_p = 0$.
	Thus $\slim_{k \rightarrow \infty} \norm{\phi_k - \langle g_n \rangle}_p = 0$, and the proof is concluded.
\end{proof}

It is also possible to define a duality pairing on the spaces $\sL^p$.

\begin{definition}\label{def strong scalar product}
	For every measurable $A \subset \civita$ and for every $1\leq p \leq q \leq \infty$, if $1/p+1/q=1$ (allowing also for $p=1$ and $q=\infty$), define an operator $\int_A : \sL^p(A) \times \sL^q(A) \rightarrow \civita$ by posing
	$$
		\int_A \langle f_n \rangle \langle g_n \rangle = \slim \int_A f_n g_n.
	$$
\end{definition}

The operator is well-defined, i.e.\ it is independent on the choice of the representatives for $\langle f_n \rangle$ and for $\langle g_n \rangle$.

\begin{lemma}
		For every measurable $A \subset \civita$, for every $1\leq p \leq q \leq \infty$ with $1/p+1/q=1$ (allowing also for $p=1$ and $q=\infty$), for every $\langle f_n\rangle \in \sL^p(A)$ and for every, $\langle g_n \rangle \in \sL^q(A)$, if $\langle f_n \rangle = \langle F_n \rangle$ and if $\langle g_n \rangle = \langle G_n \rangle$, then
	$$
		\int_A \langle f_n \rangle \langle g_n \rangle
		=
		\int_A \langle F_n \rangle \langle G_n \rangle.
	$$
\end{lemma}
\begin{proof}
	Let us calculate
	$$
		\int_A \langle f_n \rangle \langle g_n \rangle - \int_A \langle F_n \rangle \langle G_n \rangle
		=
		\slim \int_A f_n g_n
		-
		\slim \int_A F_n G_n.
	$$
	Notice that
	\begin{equation}\label{eq lemma independence scalar product}
	\int_A F_n G_n = \int_A (F_n - f_n) G_n + \int_A f_n G_n
	\end{equation}
	By the H\"older's inequality,
	$$
		\int_A |(F_n - f_n) G_n | \leq \norm{F_n - f_n}_p \norm{G_n}_q
	$$
	and, since $\slim_{n \rightarrow \infty} \norm{F_n - f_n}_p = 0$, then by the squeeze theorem for strong convergence we also have $\slim_{n\rightarrow \infty} \int_A |(F_n - f_n) G_n | = 0$.
	Since $\left| \int_A (F_n-f_n) G_n \right| \leq \int_A |(F_n - f_n) G_n |$, we obtain also $\slim_{n\rightarrow \infty} \int_A (F_n - f_n) G_n  = 0$.
	From equation \eqref{eq lemma independence scalar product}, we conclude that $\int_A F_n G_n = \int_A f_n G_n$.
	
	With a similar argument, it is possible to deduce also $\int_A f_n G_n = \int_A f_n g_n$, as desired.
\end{proof}

Thanks to the properties of the integral over $\civita$ and of the strong convergence, it is possible to prove that $\int$ is an inner product over $\sL^2$.

\begin{lemma}\label{lemma scalar product}
	For every measurable $A \subset \civita$, $\int$ is a scalar product over $\sL^2(A)$.
\end{lemma}
\begin{proof}
	Symmetry and bilinearity can be obtained from the definition of $\int$.
	In order to prove that $\int \langle f_n \rangle \langle f_n \rangle > 0$ whenever $\langle f_n \rangle \ne 0$, recall that $\int \langle f_n \rangle \langle f_n \rangle = \norm{\langle f_n\rangle}_2^2$ and that, as a consequence of Lemma \ref{lemma norm p}, if $\langle f_n \rangle \ne 0$ then $\norm{\langle f_n\rangle}_2 > 0$.
\end{proof}

As a consequence of the sequential completeness of $\sL^2$ and of the existence of the scalar product $\int$, we conclude that $\sL^2$ is a Hilbert space.

\begin{proposition}
	For every measurable $A \subset \civita$, the space $\sL^2(A)$ is a Hilbert space.
\end{proposition}
\begin{proof}
	In Proposition \ref{proposition banach spaces} we have proved that $\sL^2(A)$ with the scalar product $\int$ is a Banach space.
	Since $\int_A$ is a scalar product over $\sL^2(A)$, as shown in Lemma \ref{lemma scalar product}, $\sL^2(A)$ is a Hilbert space.
\end{proof}

The previous results show that the $\sL^p$ spaces share many good properties with the real $L^p$ spaces based upon the Lebesgue measure.
Despite these positive features, however, they are still not rich enough to represent real functions that are not locally analytic.
More precisely, we will prove that for every non-negative real measurable function $f$ over the interval $[a,b]$, if $f$ is not locally analytic almost everywhere, then there exist no $\phi \in \L_s^p([a,b]_\civita)$ such that
\begin{equation}\label{eqn coherence}
	\int_I |f(x)|^p dx \approx \int_{I_\civita} |\phi|^p \text{ for every interval } I \subseteq [a,b].
\end{equation}

\begin{proposition}\label{prop slp non rappresenta f reali}
	If $f \in L^p([a,b])$ is a non-negative function that is not locally analytic at any point in $[a,b]$, then
	for every $\langle g_n \rangle \in \sL^p([a,b]_\civita)$ there exists an interval $I(c,d)$ with $c, d \in \R$ and $c\ne d$ such that
	$$\norm{\langle g_n \cdot \chi_{I(c,d)} \rangle}_p
	\not\approx
	\norm{ |f(x)| \cdot \chi_{[c,d]}}_p .$$
\end{proposition}
\begin{proof}
	By Proposition \ref{proposition banach spaces}, $\L^p([a,b]_\civita)$ is dense in $\sL^p([a,b]_\civita)$.
	As a consequence, it is sufficient to prove that
	for every $g \in \L^p([a,b]_\civita)$ there exists an interval $I(c,d)$ with $c, d \in \R$ and $c\ne d$ such that
	$$\norm{ g \cdot \chi_{I(c,d)} }_p
	\not\approx
	\norm{ |f(x)| \cdot \chi_{[c,d]}}_p .$$
	
	If $\lambda(g(x))<0$ for every $x \in [a,b]_\civita$, then the desired inequality is trivial, since in this case $\lambda \left(\norm{ g \cdot \chi_{I} }_p \right)< 0$.
	Suppose then that there exists $I \subseteq [a,b]_\civita$ with $l(I) \not \sim 0$ such that $g$ is analytic over $I$ and $\lambda(g(x))\geq 0$ for every $x \in I$.
	Then it is possible to define a real function $G : I\cap\R \rightarrow \R$ by posing $G(r) = g(r)[0]$ for every $r \in I\cap\R$.
	Moreover, by our choice of $I$, $G$ is analytic in $I\cap\R$.
	As a consequence, by Lemma \ref{lemma coherence}, $\norm{G \cdot \chi_{J\cap\R}}_p \approx \norm{g \cdot \chi_{J}}_p$ for every interval $J \subseteq I$.
		However, since $f$ is not locally analytic at every point of its domain, there exist $c, d \in I$ such that
	$$
		\norm{G \cdot \chi_{[c,d]}}_p \ne \norm{f \cdot \chi_{[c,d]}}_p
	$$
	for every $1 \leq p \leq \infty$.
	This is sufficient to conclude the proof.
\end{proof}

\begin{remark}
	We have chosen to formalize the notion of \emph{representation of a real function by an element of $\L_s^p$} with condition \eqref{eqn coherence} since for any pair of non-negative functions $f, g \in L^p([a,b])$ the property
	$$
	\int_I |f(x)|^p dx = \int_{I} |g(x)|^p \text{ for every interval } I \subseteq [a,b]
	$$
	is equivalent to $f = g$ a.e.\ in $[a,b]$.
\end{remark}

\section{Representing real functions and distributions}\label{section representation}

We have seen in Proposition \ref{prop nonmeas} that both the space of measurable functions and the $\L^p$ spaces over the Levi-Civita field are not expressive enough to represent real continuous functions.
Similarly, Proposition \ref{prop slp non rappresenta f reali} entails that it is not possible to represent real functions that are not locally analytic as strongly converging sequences of $\L^p$ functions.
However, it is possible to represent a large class of real functions as weakly Cauchy sequences of measurable functions.
This, in turn, will enable the representation of real distributions with measurable functions on the Levi-Civita field.

\subsection{Representing real functions as weakly Cauchy sequences of $\L^p$ functions}
Even if it is not possible to represent real continuous functions with elements of $\sL^p$, weakly Cauchy sequences of $\L^p$ functions are expressive enough to represent not only real continuous functions, but every real function whose $p$-th power is summable.

\begin{theorem}\label{mainthm extension}
	Let $I \subset \R$ be an interval.
	For every real $1 \leq p \leq \infty$ and for every $f \in L^p(I)$, there exists a weakly Cauchy sequence $\{\s_n\}_{n\in\N}$ of measurable functions in $\meas(I_\civita)$ such that
	\begin{enumerate}
		\item\label{mainthm1} $\wlim_{n \rightarrow \infty} \norm{\s_n}_p = \norm{f}_p$;
		\item\label{mainthm3} for every interval $J \subseteq I$, $\wlim_{n \rightarrow \infty} \norm{ \s_n \cdot \chi_{J_\civita}}_p = \norm{f \cdot \chi_J}_p$; and
		\item\label{mainthm2} for a.e.\ $x \in I$ there exists a subsequence $\{\s_{n_k}\}_{k\in\N}$ such that $\wlim_{k \rightarrow \infty} \s_{n_k}(x) = f(x)$.
	\end{enumerate}
\end{theorem}
\begin{proof}
	Let $f \in L^p(I)$ and let $\{s_n\}_{n\in\N}$ be a sequence of simple functions satisfying $\lim_{n\rightarrow\infty} \norm{s_n-f}_p =0$.
	Since $s_n$ is simple for every $n \in\N$, there exists $k_n$ such that $s_n = \sum_{i=1}^{k_n} s_{n,i}$, with $s_{n,i}$ constant over an interval $I_{n,i} \subseteq I$.
	Then, if we let $\s_{n,i}$ be the canonical extension of $s_{n,i}$ to $(I_{n,i})_\civita$, we can define
	$
		\s_n = \sum_{i=1}^{k_n} \s_{n,i}.
	$
	The functions $\s_{n,i}$ are measurable by Lemma \ref{lemma coherence}, so that $\s_n$ is also measurable.
	Moreover, the same Lemma entails that
	\begin{equation}\label{equazione uguaglianza norme mainthm}
		\norm{\s_n}_p = \norm{\sum_{i=1}^{k_n} \s_{n,i}}_p
		=
		\norm{\sum_{i=1}^{k_n} s_{n,i}}_p
		=
		\norm{s_n}.
	\end{equation}

	This equality, the hypothesis that $\lim_{n \rightarrow \infty} \norm{s_n- f}_p = 0$ and Corollary \ref{corollario standard} entail that
	\begin{itemize}
		\item $\{\s_n\}_{n\in\N}$ is weakly Cauchy in $\L^p(I_\civita)$ and
		\item $\wlim_{n \rightarrow \infty} \norm{\{\s_n\}_{n\in\N}}_p = \norm{f}_p$.
	\end{itemize}
	This is sufficient to prove \eqref{mainthm1}.
	The proof of equality \eqref{mainthm3} can be obtained by localizing the above proof to any interval $J \subseteq I$.

	Property \eqref{mainthm2} can be obtained from Corollary \ref{corollario standard} and from the hypothesis that $s_n$ converges to $f$ in the $L^p$ norm.
	The latter entails that for a.e.\ $x \in I$ there exists a subsequence $\{s_{n_k}\}_{k\in\N}$ such that $\lim_{k\rightarrow\infty} s_{n_k}(x) = f(x)$.
	Thus by Corollary \ref{corollario standard} we also have $\wlim_{k\rightarrow\infty} \s_{n_k}(x) = f(x)$.
\end{proof}

As a consequence of the above theorem, we can represent $L^p$ functions as weakly Cauchy sequences of measurable functions.

\begin{remark}\label{remark nonorm}
	Theorem \ref{mainthm extension} might suggest that it could be convenient to define not only the spaces $\sL^p$, but also some $\wL^p$ spaces obtained by substituting strong limits with weak limits in Definitions \ref{def pre slp} and \ref{def slp}.
	However, this route presents many obstacles.
	The principal is that weak convergence does not satisfy a squeeze theorem, so many results of Section \ref{section completion} turn out to be false for the $\wL^p$ spaces.
	For instance, if $\{f_n\}_{n\in\N}$ is a weakly Cauchy sequence such that $\wlim_{n \rightarrow \infty} \norm{f_n}_p \in \civita$, one might define $\norm{ \langle f_n \rangle }_p = \wlim_{n \rightarrow \infty} \norm{f_n}_p$.
	However, with this definition there would be sequences $\{f_n\}_{n\in\N}$ satisfying $\norm{f_n}_p > 0$ for all $n \in\N$, but $\norm{ \langle f_n \rangle }_p < 0$ (one of such sequences can be obtained from Example \ref{example convergence x norm}).
	Consequently, the map $\norm{ \cdot }_p$ would not be a norm or even a seminorm.
	
	Moreover, the value $\norm{ \langle f_n \rangle }_p$ would not be independent from the representative $\{f_n\}_{n\in\N}$.
	This happens because there are sequences of measurable functions $\{f_n\}_{n\in\N}$ and $\{g_n\}_{n\in\N}$ such that
	$\wlim_{n \rightarrow \infty} \norm{ f_n - g_n }_p = 0$,
	but	$\wlim_{n \rightarrow \infty} \norm{f_n}_p \ne \wlim_{n \rightarrow \infty} \norm{g_n}_p$.
	
	A possible workaround could be that of defining a function $\norm{ \{f_n\}_{n\in\N} }_p = \lim_{n\rightarrow \infty} \sh{\norm{f_n}_p}$ and in identifying two sequences $\{f_n\}_{n\in\N}$ and $\{g_n\}_{n\in\N}$ whenever $\lim_{n \rightarrow \infty} \sh{\norm{ f_n - g_n }_p}=0$.
	However, such a definition would entail the identification of all functions whose values differ up to an infinitesimal.
\end{remark}

We remark that some classes of measurable functions could have a more meaningful representation than the one obtained in the proof of Theorem \ref{mainthm extension}.
The strongest result can be obtained for real continuous functions defined over compact intervals.

\begin{proposition}\label{proposition extension of continuous functions}
	For every $f\in C^0([a,b])$ and for every $n\in\N$, there exists a sequence of measurable functions $\{\ext{n}{p}\}_{n\in\N}$ that satisfies the following properties for every $1 \leq p \leq \infty$:
	\begin{enumerate}
		\item\label{contfun1} $\{\ext{n}{p}\}_{n\in\N}$ is weakly Cauchy in $\L^p([a,b]_\civita)$;
		\item\label{contfun2} $\wlim_{n \rightarrow \infty} \norm{\ext{n}{p}}_p = \norm{f}_p$;
		\item\label{contfun3} $\wlim_{n \rightarrow \infty} \ext{n}{p}(x) = f(x)$ for every $x \in [a,b]$.
	\end{enumerate}
\end{proposition}
	\begin{proof}
	Let $[a,b]\subset\R$ and let $\rP_n([a,b])$ the space of real polynomials of degree at most $n$ defined over $[a,b]$:
	$$\rP_n([a,b])=\left\{ p: [a,b] \rightarrow\R : \exists a_0, \ldots, a_n \in \R \text{ such that } p(x)=\sum_{i=0}^n a_n x^n\right\}.$$ 
	Recall that for every $f\in C^0([a,b])$ and for every $n\in\N$, there exists a unique $\p{n}{f} \in \rP_n$ such that
	\begin{equation}\label{limite approx 1}
		\norm{f-\p{n}{f} }_\infty \leq \norm{f-p}_\infty \text{ for all } p \in \rP_n
	\end{equation}
	and
	\begin{equation}\label{limite approx 2}
		\lim_{n \rightarrow \infty} \norm{\p{n}{f}}_p  = \norm{f}_p
	\end{equation}
	for every $1 \leq p \leq \infty$.
	For a proof of these statements, we refer for instance to \cite{best approx}.
	
	For every $n \in\N$ define $\ext{n}{p}$ as the canonical extension of the polynomial ${\p{n}{f}}$ to $[a,b]_\civita$.
	Properties \eqref{contfun1} and \eqref{contfun2} can be obtained from equation \eqref{limite approx 2}, from Lemma \ref{lemma coherence} and from Corollary \ref{corollario standard}: since
	$$
	\int_{[a,b]_\civita} |\ext{n}{p}(x)|^p =  \int_{[a,b]} \left|\p{n}{f}(x)\right|^p dx
		\text{ and } 
	\lim_{n \rightarrow \infty} \int_{[a,b]} \left|\p{n}{f}(x)\right|^p \ dx = \int_{[a,b]} |f(x)|^p \ dx,
	$$
	then $\wlim_{n \rightarrow \infty} \norm{\ext{n}{p}}_p = \norm{f}_p$.
	We deduce that the sequence $\{\ext{n}{p}\}_{n\in\N}$ is weakly Cauchy in $\L^p([a,b]_\civita)$
	and that
	$\slim_{n \rightarrow \infty} \norm{\ext{n}{p}}_p = \norm{f}_p$.
	
	Property \eqref{contfun3} is a consequence of the equality $\ext{n}{p}(x) = \p{n}{f}(x)$ for all $x \in [a,b]$, of uniform convergence of $\p{n}{f}$ to $f$ and of Corollary \ref{corollario standard}.
\end{proof}

For reciprocals of simple functions that diverge at a point, like the ones discussed in Proposition \ref{prop nonmeas}, it might be more convenient to work with a different approximation, as shown in the next example.

\begin{example}
	Let $a > 1$ and consider the function $f(x) = x^{-\frac{1}{a}}$ for $x \in (0,1]_\civita$.
	Recall that, by Proposition \ref{prop nonmeas}, $f(x)$ is not measurable over $(0,1]_\civita$.
	
	Define the functions $f_n = f_{|[n^{-1},1]_\civita}$ for every $n \in \N$.
	By Lemma \ref{lemma coherence},
	$$
	\int |f_n|^p = \int_{n^{-1}}^1 x^{-\frac{p}{a}} dx.
	$$
	It is well-known that the real sequence $n \mapsto \int_{n^{-1}}^1 x^{-\frac{p}{a}} dx$ converges for every $1 \leq p < a$: consequently, under these conditions over $p$, we have
	\begin{itemize}
		\item $\{f_n\}_p$ is weakly Cauchy in $L^p((0,1]_\civita)$;
		\item $\wlim_{n \rightarrow \infty} \norm{ f_n }_p = \norm{f}_p$.
	\end{itemize}
	In addition, we have $\wlim_{n \rightarrow \infty} f_n(x) = f(x)$ for every $x \in (0,1]_\civita$, $x \not \in M_o$ (i.e.\ for every nearstandard $x \in (0,1]_\civita$).
\end{example}

\subsection{Dirac-like measurable functions}\label{sec distributions}

The Dirac distribution centred at $r \in \R$
is a continuous linear functional defined over $C^0(\R)$ by the formula
$$
	\langle D_r, \varphi \rangle = \varphi(r)
$$
for every $\varphi\in C^0(\R)$ (where $\langle \cdot, \cdot, \rangle$ denotes the duality between a distribution and a test function).

Measurable functions that represent the Dirac distribution have already been studied by Shamseddine and Flynn \cite{shamflin2}.
The authors defined some specific representatives of the Dirac distributions and applied them for the description of physical phenomena. For instance, they discussed the measurable solutions to the Poisson equation $y''(t)=f(t)$ and determined the Green function of a damped driven harmonic oscillator.

In this section we will define a family of measurable functions over the Levi-Civita field that represent the Dirac distribution, and study some of their theoretical properties.
We will also discuss the product between the Dirac and the Heaviside distribution, a well-known problem of the nonlinear theory of distributions (for a discussion of this problem and for a solution in an algebra of generalized functions, we refer to \cite{grid functions}).
This problem cannot be treated in the space of real distributions; instead it requires algebras of generalized functions, such as Colombeau algebras \cite{colombeau}, asymptotic functions \cite{todorov}, ultrafunctions \cite{ultrafunctions} and grid functions \cite{grid functions}.
Notice that these algebras have non-Archimedean rings of scalars and, if one exclude Colombeau algebras, the ring of scalar of the other algebras of generalized functions are fields obtained with techniques of Robinson's analysis with infinitesimals.
The non-Archimedean rings of scalars underlying these spaces of generalized functions suggests that nonlinear theories of distributions require a non-Archimedean setting.
Inspired by this observation, we aim at  proving that that the results concerning the product between the Dirac and the Heaviside distribution obtained with these algebras of generalized functions can also be obtained with measurable functions on on the Levi-Civita field.

The Dirac measurable functions discussed by Shamseddine and Flynn are non-negative measurable functions, their support is included in an infinitesimal neighbourhood of a point and that their integral equals $1$.
We will say that a measurable function with these properties is a Dirac-like function.

\begin{definition}\label{def dirac-like}
	A measurable function $\delta_r \in \L^1(A)$ is Dirac-like at $r \in A$ iff
	\begin{enumerate}
		\item $\delta_{r}(x) \geq 0$ for all $x \in A$;
		\item there exists $h\in M_o$, $h>0$ such that $\supp \delta_{r} \subseteq [r-h,r+h]_{\civita} \subseteq A$;
		\item $\norm{\delta_r}_1  = 1$.
	\end{enumerate}
\end{definition}

By using the representation of real continuous functions of Proposition \ref{proposition extension of continuous functions}, we will now argue that every Dirac-like function represents the real Dirac distribution.

\begin{proposition}\label{prop dirac-like}
	Let $r \in [a,b]$, $r \ne a$ and $r \ne b$.
	For all Dirac-like measurable functions $\delta_r \in \L^1([a,b]_\civita)$ and for all $f \in C^0([a,b])$,
	if $\ext{n}{p}$ is defined as in Proposition \ref{proposition extension of continuous functions},
	then
	$$
	\lim_{n\rightarrow \infty} \sh{\left( \int_{[a,b]} \delta_r \cdot \ext{n}{p}\right)}
	=
	\lim_{n\rightarrow \infty} \left( \int_{[a,b]} \delta_r \cdot \ext{n}{p}\right)[0]
	=
	f(r).$$
\end{proposition}
\begin{proof}
	If $r \not \in \supp f$, then by uniform convergence of $\p{n}{f}$ to $f$ in $[a,b]$, we have also $\wlim_{n \rightarrow \infty} \ext{n}{p} = 0$ for almost every $x \sim r$, $x \in [a,b]_\civita$.
	As a consequence, the desired assertion is true.
	
	Suppose then that $r \in \supp f$ and let $h \in M_o$ such that $\supp \delta_r \subseteq [r-h,r+h]_\civita$.
	Since $\ext{n}{p}$ is analytic over $[a,b]_\civita$, for all $n \in \N$ the numbers
	\begin{eqnarray*}
	m_n = \min_{x \in [r-h,r+h]} \ext{n}{p}(x)
	&
	\text{and}
	&
	M_n = \max_{x \in [r-h,r+h]} \ext{n}{p}(x)
	\end{eqnarray*}
	are well-defined \cite{evt}.
	As a consequence of these inequalities and thanks to Corollary 4.6 of \cite{berz+shamseddine2003},
	\begin{equation}\label{inequalities}
	m_n = m_n \int_{[r-h,r+h]_\civita} \delta_r \leq \int_{[r-h,r+h]_\civita} \delta_r \ext{n}{p}  \leq M_n \int_{[r-h,r+h]_\civita} \delta_r = M_n
	\end{equation}
	for all $n \in \N$.
	Thus it is sufficient to prove that $\lim_{n \rightarrow \infty} M_n[0] = \lim_{n \rightarrow \infty} m_n[0] = f(r)$.
	The desired result is a consequence of the equalities $M_n[0]=m_n[0]=\ext{n}{p}(r) = \p{n}{f}(r)$ and of Proposition \ref{proposition extension of continuous functions}. 
\end{proof}

\subsection{Measurable representatives of the Heaviside distribution}\label{sec heaviside}

In principle, the Heaviside distribution could be represented by any non decreasing measurable function $H$ such that $H(x) = 0$ whenever $x < 0$, $x \not \in M_o$ and $H(x)=1$ whenever $0<x$, $x \not \in M_o$.
However, it will be more convenient to define it as the antiderivative of a continuous Dirac-like distribution centred at $0$.

\begin{definition}\label{def heaviside}
	Let $\delta_0$ be a continuous
	Dirac-like measurable function.
	Define $H : \civita \rightarrow \civita$ as 
	$$H(x) = \int_{[-1,x]} \delta_0.$$
\end{definition}

We remark that every continuous Dirac-like measurable function yields a corresponding function $H$. With a slight abuse of notation, we will not explicitly denote the dependence of $H$ upon $\delta_0$.

The functions obtained with Definition \ref{def heaviside} represent the Heaviside distribution.

\begin{proposition}
	For all $f \in C^0([a,b])$,
	if $\ext{n}{p}$ is defined as in Proposition \ref{proposition extension of continuous functions},
	then
	$$
	\lim_{n\rightarrow \infty} \sh{\left(\int_{[a,b]} H \cdot \ext{n}{p} \right)}
	=
	\lim_{n\rightarrow \infty} \left(\int_{[a,b]} H \cdot \ext{n}{p} \right)[0]
	=
	\int_0^{+\infty} f(x)\ dx.$$
\end{proposition}
\begin{proof}
	If $f \in C^0([a,b])$ and if $H_\R$ is the restriction of $H$ to $\R$, then $H_\R \cdot f$ is equal to the restriction of $f$ to $[a,b] \cap (0,\infty)$.
	
	Let $h\in M_o$, $h>0$ satisfy $\supp(\delta_0)\subseteq [-h,h]_\civita$.
	Then
	$$
	\int_{[a,b]} H \cdot \ext{n}{p}
	=
	\int_{[a,h]} H \cdot \ext{n}{p}
	+
	\int_{[h,b]} \ext{n}{p}.
	$$
	Since $h \in M_0$ and since both $h$ and $\ext{n}{p}$ are bounded by some real number $r \in \R$,
	$$
	\left| \int_{[a,h]} H \cdot \ext{n}{p} \right| \leq h \cdot r \approx 0.
	$$
	Moreover, $H(x) \cdot \ext{n}{p}(x) = \p{n}{f}(x)$ for every $x \in [a,b]_\civita \cap [h,\infty)$.
	Thus, by Lemma \ref{lemma coherence},
	$$
	\sh{\left(\int_{[h,b]} H \cdot \ext{n}{p} \right)}
	=
	\int_{[0,b]} \p{n}{f}\ dx.
	$$
	The desired equality can then be obtained from uniform convergence of $\left\{\p{n}{f}\right\}_{n\in\N}$ to $f$ over $[a,b] \cap [0,\infty)$ with an argument similar to the one used in the proof of Theorem \ref{mainthm extension}.
\end{proof}

One of the many advantages of working with representatives of the Heaviside distribution defined over non-Archimedean domains is that they allow for some calculations that are not possible with the real distributions.
For instance, in the description of shock waves Colombeau works with a representative of the Heaviside distribution and of the Dirac distribution that satisfy the equality
\begin{equation}\label{eqn Colombeau}
\int_{\R} (H^m-H^n)\cdot \delta_0 = \frac{1}{m+1}- \frac{1}{n+1}.
\end{equation}
As recalled in \cite{grid functions}, this calculation is not justified in the theory of distributions.
A first obstacle is that, in the real case, the powers $H^n$ of the Heaviside distribution are all equal to the original function $H$.
However, this does not happen for the measurable representatives introduced in Definition \ref{def heaviside}.

\begin{lemma}\label{lemma heaviside}
	For every $n \in \N$, and for every $f \in C^0([a,b])$,
	$$
	\lim_{n\rightarrow \infty} \sh{\left(\int_{[a,b]} H^n \cdot \ext{n}{p} \right)}
	=
	\lim_{n\rightarrow \infty} \left(\int_{[a,b]} H^n \cdot \ext{n}{p} \right)[0]
	=
	\int_0^{+\infty} f(x)\ dx.$$
	However, $H^n = H^m$ if and only if $n = m$.
\end{lemma}
\begin{proof}
	Let $h \in M_o$ such that $\supp \delta_0 \subseteq [-h,h]_\civita$.
	Then for every $m, n \in \N$
	\begin{enumerate}
		\item\label{1lemma} if $x \leq -h$, then $H^n(x) = H^m(x) = 0$.
		\item\label{2lemma} if $x \geq h$ then $H^n(x)=H^m(x) = 1$
		\item\label{3lemma} there exists $x \in [-h,h]$ such that $0 < H(x) < 1$.
	\end{enumerate}
	By point \eqref{3lemma}, we immediately conclude that $H^n = H^m$ if and only if $n = m$.

	Thanks to properties \eqref{1lemma} and \eqref{2lemma}, if $b< -h$ or $a > h$ the equality $\lim_{n\rightarrow \infty} \left(\int_{[a,b]} H^n \cdot \ext{n}{p} \right)[0] = \int_0^{+\infty} f(x)\ dx$ is trivially true.
	
	Suppose then that $a \leq h \leq b$. In this case, we can write
	$$
		\int_{[a,b]} H^n \cdot \ext{n}{p}
		=
		\int_{[a,h]} H^n \cdot \ext{n}{p}
		+
		\int_{[h,b]} \ext{n}{p}.
	$$
	Since $h \in M_0$, by Lemma \ref{lemma coherence} and by equation \eqref{limite approx 2} we have
	$$
		\lim_{n\rightarrow \infty} \sh{\left(\int_{[h,b]} \ext{n}{p} \right)}
		=
		\lim_{n\rightarrow \infty} \left(\int_{[h,b]} \ext{n}{p} \right)[0]
		=
		\int_0^{+\infty} f(x)\ dx.
	$$
	Similarly, since $f$ is bounded and since $0 \leq H^n(x) \leq 1$ for every $n \in\N$ and for every $x \in \civita$,
	$$
		\int_{[a,h]} H^n \cdot \ext{n}{p} \approx 0.
	$$
	We conclude
	$$
		\lim_{n\rightarrow \infty} \left(\int_{[a,b]} H^n \cdot \ext{n}{p} \right)[0]
		=
		\lim_{n\rightarrow \infty} \left(\int_{[h,b]} \ext{n}{p} \right)[0]
		+
		\lim_{n\rightarrow \infty} \left(\int_{[a,h]} H^n \cdot \ext{n}{p} \right)[0]
		=
		\int_0^{+\infty} f(x)\ dx.
	$$
	
	The case where $a \leq -h \leq b$ can be treated in a similar way.
\end{proof}

The second obstacle for the evaluation of the expression \eqref{eqn Colombeau} is that the product between the Heaviside distribution and the Dirac distribution is not defined.
However, the measurable representatives of the Dirac and of the Heaviside distribution allow to evaluate the product $H \cdot \delta_0$, similarly to what happens in Colombeau algebras \cite{colombeau} and in spaces of generalized functions of nonstandard analysis \cite{grid functions, todorov}.

\begin{proposition}
	For all $f \in C^0([a,b])$ and for all $1 \leq p \leq \infty$,
	if $\ext{n}{p}$ is defined as in Proposition \ref{proposition extension of continuous functions},
	then
	$$
	\lim_{n\rightarrow \infty} \sh{\left( \int_{[a,b]} (H \cdot \delta_0) \cdot \ext{n}{p} \right)}
	=
	\lim_{n\rightarrow \infty} \left( \int_{[a,b]} (H \cdot \delta_0) \cdot \ext{n}{p} \right)[0]
	=
	\frac{1}{2} f(0)
	$$
	and
	$$
		\int_{[a,b]} (H^m-H^n)\cdot \delta_0 = \frac{1}{m+1}-\frac{1}{n+1}.
	$$
\end{proposition}
\begin{proof}
	Recall that we have supposed that $\delta_0$ is continuous: thanks to Proposition \ref{proposition fundamental theorem of calculus}, $H \in C^1(\civita)$ and $H'(x) = \delta_0(x)$ for a.e.\ $x \in \civita$.
	Recall also that the well-known product formula for the derivative is still true for differentiable functions defined on the Levi-Civita field \cite{calculusnumerics}: consequently, we have
	$(H^{m+1})'=(m+1)H^m\cdot \delta_0$ for all $m \in\N$.
	Taking $m=1$, we deduce that $\int_{[a,b]} (H \cdot \delta_0) \cdot \ext{n}{p} = \frac{1}{2}\int_{[a,b]} (H^2)' \cdot \ext{n}{p}$.
	Now notice that
	\begin{enumerate}
		\item $(H^2)'(x)\geq 0$ for every $x \in \civita$;
		\item there exists $h \in M_o$, $h>0$ such that $\supp (H^2)' \subseteq [-h,h]$;
		\item by definition of the integral, $\int_{[-h,h]} (H^2)' = H^2(h)-H^2(-h)$.
	\end{enumerate}
	We have already observed in the proof of Lemma \ref{lemma heaviside} that $H(h)=1$ and $H(-h)=0$, so that also $H^2(h)=1$ and $H^2(-h)=0$.
	Thus $(H^2)'$ is a Delta-like measurable function in the sense of Definition \ref{def dirac-like}.
	Proposition \ref{prop dirac-like} ensures that
	$$
	\lim_{n\rightarrow \infty} \sh{\left( \frac{1}{2}\int_{[a,b]} (H^2)' \cdot \ext{n}{p} \right)}
	=
	\lim_{n\rightarrow \infty} \left( \frac{1}{2}\int_{[a,b]} (H^2)' \cdot \ext{n}{p} \right)[0]
	=
	\frac{1}{2} f(0), 
	$$
	as desired.
	
	In order to prove the second equality, let $h \in M_o$ such that $\supp \delta_0 \subseteq [-h,h]$.
	Since $(H^{m+1})'=(m+1)H^m\cdot \delta_0$,
	\begin{eqnarray*}
	\int_{[a,b]} (H^m-H^n)\cdot \delta_0
	&=&
	\int_{[-h,h]} \left(\frac{H^{m+1}}{m+1}-\frac{H^{n+1}}{n+1}\right)'\\ \\
	&=&
	\left(\frac{H^{m+1}(h)}{m+1}-\frac{H^{n+1}(h)}{n+1}\right) - \left(\frac{H^{m+1}(-h)}{m+1}-\frac{H^{n+1}(-h)}{n+1}\right).
	\end{eqnarray*}
	Since $H^{m}(-h) = H^n(-h) = 0$ and $H^{m}(h) = H^n(h) = 1$ for all $m,n \in \N$, we obtain the desired result.
\end{proof}

The previous Proposition shows that, whenever $\delta_0$ is a Dirac-like continuous and measurable function, the product $H\cdot\delta_0$ is well-defined and is equal to $\frac{1}{2}\delta_0$.
This result agrees with similar calculations obtained with other generalized functions.
Moreover, the evaluation of the integral \eqref{eqn Colombeau} is justified also in the Levi-Civita field.
For a more detailed discussion on these topics from the nonlinear theory of distributions, we refer to \cite{grid functions}.

\subsection{Derivatives of the Dirac distribution}\label{sec derivatives heaviside}

In this section, we will show how the derivatives of Dirac-like measurable functions represent the distributional derivatives of the Dirac distributions.
In order to do so, we need to establish some results on the representation of the derivative of a real continuous function by means of sequences of measurable functions in the Levi-Civita field.

Recall that, if $F \in C^1([a,b])$ and if $\{P_n\}_{n\in\N}$ uniformly converges to $F$ over $[a,b]$, then the sequence $\{P_n'\}_{n\in\N}$ might not converge to $F'$ over $[a,b]$.
This property suggests a subtler approach to the representation of the derivative of a differentiable function.

If $f \in C^0([a,b])$, denote by $F \in C^1([a,b])$ the function defined by $F(x) = \int_a^x f(t)\ dt$.
It is well-known that, if $\p{n}{f}$ are defined as in the proof of Proposition \ref{proposition extension of continuous functions} and if $P_n(x) = \int_a^x \p{n}{f}(t) \ dt$, then the sequence $\{P_n\}_{n\in\N}$ uniformly converges to $F$ over $[a,b]$.
Let $\ext{n}{P}$ be the canonical extension of $P_n$ over $[a,b]_\civita$.
Then, as a consequence of Theorem \ref{mainthm extension} and of Proposition \ref{proposition extension of continuous functions}, the sequence $\{\ext{n}{P}\}_{n\in\N}$ satisfies the conditions
\begin{enumerate}
	\item\label{contint1} $\{\ext{n}{P}\}_{n\in\N}$ is weakly Cauchy in $\L^p([a,b]_\civita)$ for every $1 \leq p \leq \infty$;
	\item\label{contint2} $\wlim_{n \rightarrow \infty} \norm{\ext{n}{P}}_p = \norm{F}_p$; 
	\item\label{contint3} $\wlim_{n \rightarrow \infty} \ext{n}{P}(x) = F(x)$ for every $x \in [a,b]$.
\end{enumerate}
From this representation of $F$ we can obtain a measurable representation of the derivative of the Dirac distribution.

\begin{proposition}\label{prop derivata dirac-like}
	Let $r \in [a,b]$, $r \ne a$ and $r \ne b$ and let $\delta_r$ be a Dirac-like measurable function of class $C^1$.
	Then for every $f \in C^0([a,b])$,
	$$
	\lim_{n\rightarrow \infty} \sh{\left( \int_{[a,b]} \delta_r' \cdot \ext{n}{P}\right)}
	=
	\lim_{n\rightarrow \infty} \left( \int_{[a,b]} \delta_r' \cdot \ext{n}{P}\right)[0]
	=
	-f(r).$$
\end{proposition}
\begin{proof}
	By the product formula for the derivative and by definition of $\ext{n}{P}$, we have
	$$
		\left(\delta_r \cdot \ext{n}{P}\right)' = \delta_r' \cdot \ext{n}{P} + \delta_r \cdot \ext{n}{p}.
	$$
	Thus
	$$
		\int_{[a,b]} \left(\delta_r \cdot \ext{n}{P}\right)'
		=
		\int_{[a,b]} \delta_r' \cdot \ext{n}{P}
		+
		\int_{[a,b]} \delta_r \cdot \ext{n}{p}.
	$$
	The hypothesis that $r\ne a$ and $r\ne b$ entail
	$$\int_{[a,b]} \left(\delta_r \cdot \ext{n}{P}\right)' = \delta_r(b) \cdot \ext{n}{P}(b) - \delta_r(a) \cdot \ext{n}{P}(a) = 0,$$
	so that
	$$
	\int_{[a,b]} \delta_r' \cdot \ext{n}{P}
	=
	-
	\int_{[a,b]} \delta_r \cdot \ext{n}{p}.
	$$
	By Proposition \ref{prop dirac-like} we obtain the desired equality.
\end{proof}

By iterating the previous result, one can obtain a general representation of the $k$-th derivative of the Dirac distribution.

\begin{corollary}\label{corollario derivate dirac}
	Let $r \in [a,b]$, $r \ne a$ and $r \ne b$ and let $\delta_r$ be a Dirac-like measurable function of class $C^k$.
	Let $F \in C^k([a,b])$ and define $f = F^{(k)}$.
	If the sequence of measurable functions $\{\ext{n}{p}\}_{n\in\N}$ represents $f$ in the sense of Proposition \ref{proposition extension of continuous functions}, and if $\{\ext{n}{P}\}_{n\in\N}$ is a sequence of simple functions that satisfy $\ext{n}{P}^{(k)} = \ext{n}{p}$ for every $n \in \N$, then
	$$
	\lim_{n\rightarrow \infty} \sh{\left( \int_{[a,b]} \delta_r^{(k)} \cdot \ext{n}{P}\right)}
	=
	\lim_{n\rightarrow \infty} \left( \int_{[a,b]} \delta_r^{(k)} \cdot \ext{n}{P}\right)[0]
	=
	(-1)^k f(r).$$
\end{corollary}

\section{Final remarks}\label{sec conclusiva}

In Section \ref{section representation} we have shown that it is possible to represent a class of real measurable functions with sequences of measurable functions in the Levi-Civita field, and to represent some real distributions with measurable functions on the Levi-Civita field.
This approach has allowed us to obtain some classic results from the nonlinear theory of distributions by working with measurable functions on the Levi-Civita field.
Moreover, Corollary \ref{corollario derivate dirac} suggests how the pointwise derivative over $\L^p(A) \cap C^k(A)$ can be used to represent the distributional derivative.

Despite these positive results, we believe that there are some obstacles towards representing every real distribution with measurable functions over $\civita$.
The most evident is that we have not been able to establish a $C^\infty$ structure over the spaces $L^p(A) \cap C^\infty(A)$.
Moreover, recall the following representation theorem for real distributions: if $T$ is a distribution with compact support, then there exists $f \in C^0(\R)$ and $k \in \N$ such that $T = f^{(k)}$, where the derivative is assumed in the sense of distributions \cite{strichartz}.
Thanks to this equality and from the approach used in Section \ref{sec derivatives heaviside} to obtain Corollary \ref{corollario derivate dirac}, if the real function $f$ is locally analytic, then it might be possible to represent the corresponding distribution $T_f$ with a measurable function defined on the Levi-Civita field (but recall that if $f$ is locally analytic, then $T_f$ can be identified with a locally analytic function).
However, recall that the distributional derivative is represented by the pointwise derivative of $C^1$ functions in $\civita$, as argued in Section \ref{sec derivatives heaviside}.
Thus, Proposition \ref{lemma not analytic} and Proposition \ref{prop slp non rappresenta f reali} entail that it is not possible to represent every real distribution with compact support in this way. In particular, if $T=f^{(k)}$ for some $f \in C^0(\R)$ that is not locally analytic at almost every point of its domain, then $T$ does not admit a representation as a measurable function on the Levi-Civita field.

If one is interested in representing distributions of an arbitrary support and whose order is not finite, i.e.\ such that
$
	T = \sum_{n \in \N} f_n^{(k_n)}
$
with $f_n \in C^0(\R)$ for every $n\in\N$ and with $\lim_{n \rightarrow \infty} k_n = +\infty$ (for a more precise definition of order of a distribution see Section 6.2 of \cite{strichartz}), then it might not even possible to use the approach proposed in Section \ref{sec derivatives heaviside}.

Theorem \ref{mainthm extension} and the above examples might suggest that it would be more convenient to use the weak limit instead of the strong limit in the definition of measurable sets and functions.
However, this idea needs to be worked out with great care. As we have seen in Remark \ref{remark nonorm}, it is not possible to define a notion of norm, and consequently also a notion of integral, by using the weak limit instead of the strong limit in the definition of the $\norm{\cdot}_p$ norm (see Lemma \ref{lemma norma ben def}).
In addition, since the weak limit does not satisfy a squeeze theorem, it might happen that $f_n (x) > g_n(x)$ for all $n \in \N$, but that $\wlim_{n \rightarrow \infty} \norm{ f_n }_p < \wlim_{n \rightarrow \infty} \norm{ g_n }_p$. As a consequence, results such as Proposition 4.4 and Corollary 4.5 of \cite{berz+shamseddine2003} would not be true.

A measure defined by replacing the strong limit with the weak limit in Definition \ref{definizione misura} would suffer from similar drawbacks.	
For instance, this ``measure'' would not be monotone, since it would be possible to prove that the monad $\mu(r)$ of a point $r \in \civita$ is measurable with measure $0$ (since it can be obtained as the weak limit of the measure of the sets $(-n^{-1},n^{-1})$.
As a consequence, for every $h \in M_o$, $h>0$ we would have the inclusion $(r-h,r+h)_\civita \subset \mu(r)$ together with the opposite inequality $0=\m(\mu(r)) < \mu((r-h,r+h)_\civita) = 2h$.

In order to overcome these obstacles and to exploit the strength of Theorem \ref{mainthm extension}, we believe that it could be convenient to introduce a uniform real-valued measure over the Levi-Civita field by adapting the L\"oeb measure construction of nonstandard analysis \cite{loeb}.
This real-valued measure might allow for a better representation of real continuous functions and of their duality with measurable functions on the Levi-Civita field.
This idea will be explored in a subsequent paper \cite{forthcoming}.

Notice however that allowing a too large class of functions to be measurable would lead to further difficulties.
For instance, if the measure allows for the existence of non-constant measurable functions with null derivative, then the Mean Value Theorem for the integral would fail, and it would not be possible to represent the distributional derivative with the pointwise derivative on the Levi-Civita field.

\subsection*{Acknowledgements}
The author is grateful to professor Khodr Shamseddine for his encouragement during an early stage of this research.
An anonymous referee provided valuable comments and references to a preliminary version of this paper.


\begin{thebibliography}{1}
	
	\bibitem{gang} T.\ Bascelli, E. Bottazzi et al. \emph{Fermat, Leibniz, Euler,
	and the gang: The true history of the concepts of limit and
	shadow},  {Notices of the American Mathematical Society}
	\textbf{61} 2014, no.\;8, 848--864.  See
	\url{http://www.ams.org/notices/201408/rnoti-p848.pdf}.
	
	\bibitem[Berz(1992)]{berz} M.\ Berz, \emph{Analysis on a Nonarchimedean Extension of the Real Numbers}. Lecture Notes, 1992.
	
	\bibitem{calculusnumerics} M.\ Berz, \emph{Calculus and numerics on Levi-Civita fields}. In M. Berz, C. Bischof, G. Corliss, and A. Griewank, editors, \emph{Computational Differentiation: Techniques, Applications, and Tools}, pages 19–35, Philadelphia, 1996. SIAM.
	
	\bibitem[Berz\&Shamseddine(2010)]{analysislcf} M. Berz, K. Shamseddine,
	\emph{Analysis on the Levi-Civita field, a brief overview},
	Contemporary Mathematics, 508, 2010, pp.\ 215-237.
	
	\bibitem[Berz\&Shamseddine(2005)]{reexp} M. Berz, K. Shamseddine, \emph{Analytical properties of power series on Levi-Civita fields}, Annales Mathématiques Blaise Pascal, Volume 12 n 2, 2005, pp.\ 309-329. 
	
	\bibitem[Berz\&Shamseddine(2000)]{computational}M. Berz, K. Shamseddine,
	\emph{The differential algebraic structure of the Levi-Civita field and applications},
	Int. J. Appl. Math., Volume 3 (2000), pp.\ 449-465.
	
	\bibitem{ultrafunctions} V. Benci, \emph{Ultrafunctions and generalized solutions}, Advanced Nonlinear Studies 13.2 (2013): 461-486.
	
	\bibitem{ultramodel} V.\ Benci, L. Luperi Baglini, \emph{A model problem for ultrafunctions}, Variational and Topological Methods: Theory, Applications, Numerical Simulations, and Open Problems (2012).

	\bibitem{benci_schwartz} V. Benci, L. Luperi Baglini, \emph{A non-archimedean algebra and the Schwartz impossibility theorem}, Monatshefte für Mathematik 176.4 (2015): 503-520.
	
	\bibitem{ultraapps} V.\ Benci, L. Luperi Baglini, \emph{Ultrafunctions and applications}, Discrete \& Continuous Dynamical Systems - Series S Vol. 7 Issue 4 (2014), p593-616. 24p. 
	
	\bibitem{forthcoming}  E. Bottazzi, \emph{A real-valued measure on non-Archimedean field extensions of $\R$}, in preparation.
	
	\bibitem{bottazzi} E. Bottazzi, \emph{A transfer principle for the continuation of real functions to the Levi-Civita field}, p-Adic Numbers, Ultrametric Analysis, and Applications, issue 3, vol 10 (2018).
	
	\bibitem{grid functions} E. Bottazzi, \emph{Grid functions of nonstandard analysis in the theory of distributions and in partial differential equations}, Advances in Mathematics 345 (2019): 429-482.
	
	\bibitem{illposed}  E. Bottazzi, \emph{A grid function formulation of a class of ill-posed parabolic equations} (2017), submitted. See \url{https://arxiv.org/abs/1704.00472}.
	
	\bibitem{prussk} E. Bottazzi, M. G. Katz, \emph{Infinite lotteries, spinners, and the applicability of hyperreals} (2020), to appear in Philosophia Mathematica.
	
	\bibitem{prusso} E. Bottazzi, M. G. Katz, \emph{Internality, transfer and infinitesimal modeling} (2020), to appear in Philosophia Mathematica.
	
	\bibitem{19c} E. Bottazzi, V. Kanovei et al. \emph{On mathematical
	realism and the applicability of hyperreals}, {Mat. Stud}.  \textbf{51} 2019, no.\;2, 200--224.  See
	\url{http://dx.doi.org/10.15330/ms.51.2.200-224}.
	
	\bibitem{best approx} J.\ G.\ Burkill, \emph{Lectures On Approximation By Polynomials}, Tata Institute of Fundamental Research (1959), ISBN/ASIN: B0007JA0CC 
	
	\bibitem{capicutland1}
	M.\ Capi\'{n}ski, N.\ J.\ Cutland, \emph{A Simple Proof of Existence of Weak and Statistical Solutions of Navier-Stokes Equations}, Proceedings: Mathematical and Physical Sciences, vol. 436, no. 1896 (1992), pp. 1–11., \url{www.jstor.org/stable/52016}.
	
	\bibitem{capicutland statistic}
	M.\ Capi\'{n}ski, N.\ J.\ Cutland, \emph{Statistical solutions of PDEs by nonstandard densities}, Monatshefte f{\"u}r Mathematik, vol. 111, no. 2 (1991), pp. 99--117, doi:\url{http://dx.doi.org/10.1007/BF01332349}.
	
	
	\bibitem{colombeau} J.F. Colombeau, \emph{Nonlinear generalized functions: their origin, some developments and recent advances}, São Paulo J. Math. Sci. (ISSN2316-9028) 7(2) (2013) 201–239.
	
	\bibitem{integral} Costin, O., Ehrlich, P., Friedman, H. M.
	\emph{Integration on the surreals: a conjecture of Conway, Kruskal
	and Norton},  preprint (2015).  See
	\url{https://arxiv.org/abs/1505.02478}.
	
	\bibitem{loeb} N.\ J. Cutland, \emph{Loeb measure theory}, in \emph{Loeb Measures in Practice: Recent Advances}, Springer Berlin Heidelberg (2000): 1--28.
	
	\bibitem{shamflin2} D.\ Flynnn, K.\ Shamseddine, \emph{On Integrable Delta Functions on the Levi-Civita Field}, p-Adic Numbers, Ultrametric Analysis and Applications, 2018, 10.1: 32-56.
	
	\bibitem{inf calc} B.\ Crowell, M.\ Khafateh, \emph{Inf}, \url{http://www.lightandmatter.com/calc/inf/}.
	
	\bibitem{hyperfinite pinto}  R.F.\ Hoskins, J.S.\ Pinto, \emph{Hyperfinite representation of distributions}, Proc. Indian Acad. Sci. (Math. Sci.) (2000) 110: 363. doi:10.1007/BF02829532
	
	\bibitem{kinoshita} M.\ Kinoshita, \emph{Non-standard representations of distributions I}, Osaka J.\ Math.\ 25 (1988), pp. 805--824.
		
	\bibitem[Levi-Civita(1892)]{civita1} T.\ Levi-Civita,
	\emph{Sugli infiniti ed infinitesimi attuali quali elementi analitici},
	Atti Ist. Veneto di Sc., Lett. ed Art., 7a (4) (1892), p. 1765.
	
	\bibitem[Levi-Civita(1898)]{civita2} T.\ Levi-Civita,
	\emph{Sui numeri transfiniti},
	Rend. Acc. Lincei, 5a (7) (1898), pp. 91-113.
	
	\bibitem{lux} W. A. J. Luxemburg, \emph{On a class of valuation fields introduced by Robinson}, Israel J. Math. 25 (1976), p.189-201.
	
	\bibitem[M{\'e}sz{\'a}ros\&Shamseddine(2015)]{odes} A.\ R.\ M{\'e}sz{\'a}ros, K.\ Shamseddine, \emph{On the solutions of linear ordinary differential equations and Bessel-type special functions on the Levi-Civita field}, Journal of Contemporary Mathematical Analysis 50, 53--62, doi=10.3103/S1068362315020016
	
	\bibitem{moreno} H.\ M.\ Moreno, \emph{Non-measurable sets in the Levi-Civita field}, in \emph{Advances in Ultrametric Analysis: 12th International Conference on P-adic Functional Analysis, July 2-6, 2012, University of Manitoba, Winnipeg, Manitoba, Canada.} American Mathematical Soc., 2013.
	
	\bibitem{todorov} M. Oberguggenberger, T. Todorov, \emph{An embedding of Schwartz distributions in the algebra of asymptotic functions}, Int. J. Math. Math. Sci. 21 (1998) 417–428.
	
	\bibitem{pestov}V. Pestov, \emph{On a valuation field invented by A. Robinson and certain
structures connected with it}, Israel J. Math. 74 1991, p. 65-79.
	
	\bibitem{pruss} A. Pruss, \emph{Underdetermination of infinitesimal
	probabilities}, {Synthese} 2018, online first at
	\url{https://doi.org/10.1007/s11229-018-02064-x}.
	
	\bibitem{robi naa} A.\ Robinson, \emph{Function theory on some nonarchimedean fields}, Amer. Math. Monthly 80 (6), Part II: Papers in the Foundations of Mathematics, 1973, p. 87-109.
	
	\bibitem{robinson} A.\ Robinson, \emph{Non-standard analysis}, Nederl. Akad. Wetensch. Proc. Ser. A 64 = Indag. Math. 23 (1961), 432–440.
	
	
	
	\bibitem{nfa} P.\ Schneider \emph{Nonarchimedean Functional Analysis}, Springer-Verlag Berlin Heidelberg 2002, DOI 10.1007/978-3-662-04728-6.
	
	\bibitem{evt} K.\ Shamseddine, \emph{Absolute and relative extrema, the mean value theorem and the inverse function theorem for analytic functions on a Levi-Civita field}, Contemp. Math 551 2011: 257-268.
	
	\bibitem{Shamseddinephd} K.\ Shamseddine, \emph{New Elements of Analysis on the Levi-Civita Field}, PhD thesis, Michigan State University, East Lansing, Michigan, USA, 1999. Also Michigan State University
	report MSUCL-1147
	
	\bibitem{shamseddine2012} K.\ Shamseddine,  \emph{New results on integration on the Levi-Civita field}, Indagationes Mathematicae, 24(1) 2013, pp.199-211.
	
	\bibitem{analytic1} K.\ Shamseddine, M.\ Berz, \emph{Analytical properties of power series on Levi-Civita fields}, Ann. Math. Blaise Pascal, 12(2) 2005, pp.309-329.
		
	\bibitem{convergence} K.\ Shamseddine, M.\ Berz, \emph{Convergence on the Levi-Civita field and study of power series}, Proc. Sixth International Conference on Nonarchimedean Analysis, pages 283–299, New
	York, NY, 2000. Marcel Dekker.
	
		
	\bibitem{berz+shamseddine2003} K.\ Shamseddine, M.\ Berz, \emph{Measure theory and integration on the Levi-Civita field}, Contemporary Mathematics, 319, 2003, pp.369-388.
	
	\bibitem{shamflin1} K.\ Shamseddine, D.\ Flynn, \emph{Measure theory and Lebesgue-like integration in two and three dimensions over the Levi-Civita field}, Contemporary Mathematics, 665, 2016, pp. 289-325.
	
	\bibitem{strichartz} R.\ Strichartz, \emph{A guide to distribution theory and Fourier Transforms}, CRC Press (1994).
	
	\bibitem{strobel} M.\ Strobel, \emph{Non-standard analysis in dynamic geometry},
	Journal of Symbolic Computation,
	Volume 97,
	2020,
	Pages 69-108,
	ISSN 0747-7171,
	\url{https://doi.org/10.1016/j.jsc.2018.12.006}.
	
	\bibitem{todorov delta} T.\ Todorov, \emph{A Nonstandard Delta Function}, Proc. Amer. Math.
	Soc., Vol. 110, Number 4, 1990, p. 1143--1144.
	
	\bibitem{todorov steady} T.\ Todorov, \emph{Steady-State Solutions in an Algebra of Generalized Functions: Lightning, Lightning Rods and Superconductivity}, Novi Sad Journal of Mathematics, Volume 45, Number 1, 2015.
\end{thebibliography}
\end{document}